\documentclass[preprint,1p]{elsarticle} 
\usepackage{amsmath,amssymb,amsfonts}
\usepackage{graphicx}
\usepackage{hyperref}
\newtheorem{theorem}{Theorem}
\newtheorem{lemma}[theorem]{Lemma}
\newdefinition{remark}{Remark}
\newproof{proof}{Proof}
\newtheorem{proposition}{Proposition}
\newtheorem{corollary}[theorem]{Corollary}

\newcommand{\R}{\mathbb{R}}

\begin{document}

\begin{frontmatter}

\title{Ground state of some variational problems in Hilbert spaces and applications to P.D.E.\tnoteref{t1}}

\tnotetext[t1]{This research project is implemented in the framework of H.F.R.I call
``Basic research Financing (Horizontal support of all Sciences)'' under the National Recovery and Resilience Plan ``Greece 2.0'' funded by the European Union - NextGenerationEU (H.F.R.I. Project Number: 016097).}

\author[inst1]{Ioannis Arkoudis}
\ead{jarkoudis@math.uoa.gr}
\author[inst2]{Panayotis Smyrnelis\corref{cor1}}
\ead{smpanos@math.uoa.gr}

\cortext[cor1]{Corresponding author}

\affiliation[inst1]{
  organization={Department of Mathematics, University of Athens},
  addressline={Panepistimiopolis},
  city={Athens},
  postcode={15784},
  country={Greece}
}

\affiliation[inst2]{
  organization={Department of Mathematics, University of Athens},
  addressline={Panepistimiopolis},
  city={Athens},
  postcode={15784},
  country={Greece}
}

\begin{abstract}
We prove the existence of a ground state for some variational problems in Hilbert spaces, following the approach of Berestycki and Lions. Next,  we examine the problem of constructing ground state solutions $u:\R^{d+k}\to\R^m$ of the system $\Delta u(x)=\nabla W(u(x))$ (with $W:\R^m\to \R$), corresponding to some nontrivial stable solutions $e:\R^k\to\R^m$. The method we propose is based on a reduction to a ground state problem in a space of functions $\mathcal H$, where $e$ is viewed as a local minimum of an effective potential defined in $\mathcal H$. As an application, by considering a heteroclinic orbit $e:\R\to\R^m$, we obtain nontrivial solutions $u:\R^{d+1}\to\R^m$ ($d\geq 2$), converging asymptotically to $e$, which can be seen as the homoclinic analogs of the heteroclinic double layers, initially constructed by Alama-Bronsard-Gui and Schatzman.

\end{abstract}


\end{frontmatter}

\section{Introduction}
 In the first part of the present paper we establish the existence of a ground state for some variational problems in Hilbert spaces, following the approach of Berestyscki and Lions in \cite{BerLionsI}. Their seminal paper, which studies nonlinear scalar fields in $\R^d$, with $d\geq3$, was completed by the work of  Berestycki,  Gallou{\"e}t, and Kavian \cite{BGK}, in the two dimensional case, and by the work of Brezis and Lieb \cite{brezis} in the vector case. For more recent developments on these results we refer for instance to the papers \cite{alves,clap,med1,med2,med3}.

In the Euclidean setting, it is known (cf. \cite{brezis}) that given a potential $W\in C^2(\R^m)$ (with $m\geq 1$) satisfying for instance the following set of assumptions:
\begin{subequations}\label{condwww}
 \begin{equation}\label{condw2}
 W(0)=0, \text{ and }W(u_0)<0, \text{ for some } u_0\in\R^m,
 \end{equation}
\begin{equation}
\begin{cases}
\text{when $d\geq 3$: } \liminf_{|u|\to 0} \frac{W(u)}{ |u|^{2^*}}\geq 0 \text{  and  } \lim_{|u|\to\infty} \frac{W(u)}{ |u|^{2^*}}=0,  \text{ with $2^*=\frac{2d}{d-2}$,}\\
\text{when $d=2$: $W(u)>W(0)=0$, for $0<|u|\leq \delta$, for some $\delta>0$,} \\  \  \  \  \  \  \  \  \  \  \  \  \  \  \  \  \  \  \  \text{and $\limsup_{|u|\to\infty} \frac{|\nabla W(u)|}{|u|^q}<\infty$, for some $q>0$,}
\end{cases}
 \end{equation}
 \end{subequations}
then the minimization problem:
\begin{equation}\label{minpb}
\begin{cases}
\text{when $d\geq 3$: }\min\{\frac{1}{2}\int_{\R^d}|\nabla u(x)|^2dx: \int_{\R^d}W(u(x))dx\leq -1\}\\
\text{when $d=2$: }\min\{\frac{1}{2}\int_{\R^2}|\nabla u(x)|^2dx: \int_{\R^2}W(u(x))dx\leq 0, \, u\not\equiv 0\}
\end{cases}
\end{equation}
admits a solution $u:\R^d\to\R^m$ in an appropriate class of maps. Moreover, under a convenient change of scale, the minimum of \eqref{minpb} gives rise to a \emph{ground state solution} for the problem
\begin{equation}\label{groundpb}
\Delta u(x)=\nabla W(u(x)) \text{ in } \R^d, \text{ with } u\not\equiv 0, \  \lim_{|x|\to\infty}u(x)=0.
\end{equation}
By a ground state solution, we mean a solution $u\not\equiv 0$ of \eqref{groundpb} such that  $E_d(u) \leq E_d(v)$ holds for every nontrivial solution $v$ of \eqref{groundpb} in an appropriate class, where we denote by
\begin{equation}\label{energy}
E_d(v)=\int_{\R^d}\big(\frac{1}{2}|\nabla v|^2+W(v)\big), \text{ or } E_d(v,\Omega)=\int_{\Omega}\big(\frac{1}{2}|\nabla v|^2+W(v)\big), \, \Omega\subset\R^d,
\end{equation}
 the energy functional associated to system \eqref{groundpb}.

Finally, we point out that the one dimensional case $d=1$ is special, and is treated differently. Under certain assumptions on $W$, which are distinct from \eqref{condwww} (cf. for instance \cite{antonop}), one obtains solutions $u:\R\to\R^m$ of the O.D.E. problem:
\begin{equation}\label{odepb}
u''(x)=\nabla W(u(x)) \text{ in } \R, \text{ with } u\not\equiv 0, \  \lim_{|x|\to\infty}u(x)=0.
\end{equation}
Such solutions are called homoclinic orbits.

In the setting of Hilbert space (cf. Theorems \ref{th1} and \ref{th2} in Section \ref{sec:groundH} below), we consider sequentially weakly lower semicontinuous potentials $\mathcal W$ defined in a separable Hilbert space $\mathcal H$ and taking their values in an interval $[-\mathcal W_m,+\infty]$, which is bounded below. In addition, we assume that $\mathcal W(0)=0$, and that $0$ is respectively a local minimum of $\mathcal W$, when $d\geq 3$, and a nondegenerate local minimum of $\mathcal W$, when $d=2$. Finally, we suppose that condition \eqref{condw2} is satisfied for $\mathcal W$. In Theorems \ref{th1} and  \ref{th2}, we establish that the minimization problem
\begin{equation}\label{minpb22}
\begin{cases}
\text{when $d\geq 3$: }\min\{\frac{1}{2}\int_{\R^d}\|\nabla U(y)\|_{\mathcal H^d}^2dy:  \int_{\R^d}\mathcal W(U(y))dy\leq -1\}\\
\text{when $d=2$: }\min\{\frac{1}{2}\int_{\R^2}\|\nabla U(y)\|_{\mathcal H^2}^2dy: \int_{\R^2} \mathcal W(U(y))dy\leq 0, \, U\not\equiv 0\}
\end{cases}
\end{equation}
admits a solution in an appropriate class of \emph{radial} maps $U:\R^d\to\mathcal H$ (cf. \eqref{class1bis} and  \eqref{class2bis}). We point out that the radial symmetry is not a restrictive constraint in the Euclidean setting, since it is proved in \cite{maris} that any minimizer of problem \eqref{minpb} is necessarily radial.  The main issue in the setting of Hilbert spaces, is due to the lower semicontinuity of the potential $\mathcal W$, and to the fact that the imbeddings of Sobolev spaces of functions $U:\R^d\supset\Omega\to\mathcal H$ are not compact. However, by assuming that $\mathcal W$ is bounded below, and looking for a minimizer in the class of radial maps, problem \eqref{minpb22} can be solved. Notice also that in contrast with the Euclidean case, none asymptotic condition is required on $\mathcal W$, as $\|U\|_{\mathcal H}\to\infty$.

The ground state solution $u$ in \eqref{groundpb} converges asymptotically to the trivial solution $0$ of system \eqref{groundpb}, and it is close to it in terms of energy. Thus, one may  ask if a ground state corresponding to $k\geq 1$ dimensional solutions can also be constructed. This is the scope of the second part of the present paper.

More precisely, given a solution $e:\R^k\to\R^m$ of
\begin{equation}\label{solhet}
\Delta e(x_{d+1},\ldots,x_{d+k})=\nabla W(e(x_{d+1},\ldots,x_{d+k})), \ (x_{d+1},\ldots,x_{d+k}) \in\R^k,
\end{equation}
we are looking for a solution $u:\R^{d+k}\to\R^m$ of system
\begin{subequations}\label{pbe}
\begin{equation}\label{solu}
\Delta u(x_{1},\ldots,x_{d+k})=\nabla W(u(x_{1},\ldots,x_{d+k})), \ (x_{1},\ldots,x_{d+k}) \in\R^{d+k},
\end{equation}
satisfying the following properties
\begin{equation}\label{propsol1}
u(x)-e(x_{d+1},\ldots,x_{d+k})\not\equiv 0  \text{ where } x=(x_1,\ldots, x_{d+k}),
\end{equation}
\begin{equation}
\lim_{|x|\to\infty}(u(x)-e(x_{d+1},\ldots,x_{d+k}))=0,
\end{equation}
\begin{equation}
\mathcal E(u)\leq\mathcal E(v) \text{ holds for every solution $v\neq e$ of \eqref{solu} in an appropriate class},
\end{equation}
\end{subequations}
where $\mathcal E$ is a renormalized energy functional related to functional $E_{d+k}$ (cf. \eqref{energy}) that will be defined in \eqref{renormene1}.

The strategy to solve problem \eqref{pbe} is inspired in the approach from Functional Analysis (evolution equations), where the solution $u$ of \eqref{pbe}, is regarded as a map $U:\R^d\to e+L^2(\R^k;\R^m)$, $[U(x_1,\ldots,x_d)](x_{d+1},\ldots,x_{d+k}):=u(x_1,\ldots, x_{d+k})$, taking its values in the affine space $e+L^2(\R^k;\R^m)$, and the initial P.D.E. problem \eqref{pbe} is reduced to the minimization problem \eqref{minpb22} for $U$. This method has previously been applied in \cite{ps2,doublehet} to construct heteroclinic double layers solutions of second and fourth order elliptic systems. We refer to the work of Alama, Bronsard, Gui \cite{abg}, and Schatzman \cite{scha} for the original construction of such solutions, which are particularly relevant for the theory of phase transition systems. Indeed, these results provide nontrivial examples of \emph{minimal solutions} for system \eqref{solu} with a double well potential $W$.

Below, we present the main ideas of this method. One issue is due to the fact that the energy of the solutions we are looking for is infinite. Thus, we have to renormalize the functional $E_k$. For this purpose, we identify the affine space $e+L^2(\R^k;\R^m)$ with the Hilbert space $\mathcal H=L^2(\R^k;\R^m)$, and define the effective potential  $\mathcal W:\mathcal H\to\R$ as follows:
\begin{equation}\label{defw}
\mathcal W(h)=\begin{cases}
\int_{\R^k}\big(\frac{1}{2}|\nabla h|^2+[W(e+h)-W(e)-\nabla W(e)\cdot h]\big), \ \text{ when } h\in H^1(\R^k;\R^m),\\
+\infty,  \text{ when } h\in L^2(\R^k;\R^m)\setminus H^1(\R^k;\R^m).
\end{cases}
\end{equation}
Assuming that the map $\R^m \ni u\mapsto \nabla W(u)$ is globally Lipschitz, the potential $\mathcal W$ is well defined. Moreover, we have $\mathcal W(0)=0$, and one can see that given a test function $\phi\in\mathcal D(\R^k;\R^m
)$, it holds $\mathcal W(\phi)=E_k(e+\phi, \operatorname{supp} \phi)-E_k(e,\operatorname{supp}\phi)$.

Assuming also that the solution $e$, which is identified with the origin of the Hilbert space $\mathcal H$,  is a local minimum of $\mathcal W$, then the minimization problem \eqref{minpb} admits a solution $V:\R^d\to\mathcal H$, provided that $\mathcal W$ is sequentially weakly lower semicontinuous and bounded below. Finally, after rescaling $V$, we obtain a map $\tilde V$ giving rise to  the solution $u(x_1,\ldots,x_{d+k})=e(x_{d+1},\ldots, x_{d+k})+[\tilde V(x_1,\ldots,x_d)](x_{d+1},\ldots,x_{d+k})$ of problem \eqref{pbe}. In addition, by deriving  a Pohozaev identity for $\tilde V$ (cf. Lemma \ref{hamid}), one can show that $u$ is a least-energy solution in an appropriate class, for the renormalized energy functional:
\begin{align}\label{renormene1}
\mathcal E(\tilde u)=\int_{\R^{d}}\big(\frac{1}{2}\|\nabla \tilde \Xi(y)\|^2_{\mathcal H^d}+\mathcal W(\tilde \Xi(y))\big)dy,
\end{align}
where $y=(x_1,\ldots,x_d)$, $\tilde u(x_1,\ldots,x_{d+k})=e(x_{d+1},\ldots, x_{d+k})+[\tilde \Xi(y)](x_{d+1},\ldots,x_{d+k})$, and $\tilde \Xi:\R^d\to L^2(\R^k;\R^m)$. Notice that the energy functional $\mathcal E$ decouples the partial derivatives $\frac{\partial v}{\partial x_1},\ldots,\frac{\partial v}{\partial x_{d+k}} $, appearing in $E_{d+k}$, since $$\int_{\R^{d}}\|\nabla \tilde \Xi(y)\|^2_{\mathcal H^d}dy=\int_{\R^{d+k}}|\nabla_y \tilde u(x)|^2 dx,$$ while in the expression of $\mathcal W$, the derivatives  $\frac{\partial (\tilde u-e)}{\partial x_i}$ ($i=d+1,\ldots,d+k)$, as well as the contribution of $W$ appear.

In Section \ref{sec:groundH}, we apply this method to a heteroclinic orbit $e:\R\to\R^m$ of the O.D.E. system
\begin{equation}\label{systode}
e''=\nabla W(e),
\end{equation}
with a multiple well potential $W\geq 0$ (cf. assumptions \eqref{w13}). That is, $e:\R\to\R^m$  is a solution of \eqref{systode} satisfying $\lim_{t\to\infty} e(t)=a^\pm$, with $W(a^\pm)=0$, and $a ^-\neq a^+$. Indeed, in contrast with other one dimensional solutions of  \eqref{systode} such as homoclinic and periodic orbits, there exist heteroclinic orbits which are stable. This is a necessary assumption to solve problem \eqref{pbe}, in Theorems \ref{th1ground} and \ref{th2ground} below. We also point out that in the two dimensional case $d=2$, we need an additional symmetry assumption\footnote{This assumption was used in the first construction of heteroclinic double layers (cf. \cite{abg}), and also in \cite{alessio}  to establish the existence of an infinity of periodic layered solutions.} on $W$ (cf. \eqref{w1ter}), to ensure that $e$ is an isolated minimum of $\mathcal W$.
Otherwise, $\mathcal W$ vanishes  along the curve $\{e^T: e^T(t):=e(t-T),  \forall T\in\R\}$
obtained by translating the heteroclinic orbit $e$, and the minimization problem \eqref{minpb22} does not have  anymore (when $d=2$) a solution distinct from $0$.

We recall that given a double well potential $W:\R^m\to[0,\infty)$, having (up to translations) exactly two minimizing heteroclinic orbits $e^\pm$ (i.e. $e^\pm$ are minimizers of $E_1$ in the affine space $e^\pm+H^1(\R;\R^m)$), the heteroclinic double layers is a solution $u:\R^2\to\R^m$ of system
\begin{equation}\label{system0}
\Delta u=\nabla W(u),
\end{equation}
such that
\begin{equation}\label{lay1}
\lim_{x_1\to\pm\infty}u(x_1,x_2)=e^\pm(x_2-m^\pm), \text{ for some constants $m^\pm\in\R$}.
\end{equation}
It is obtained in \cite{ps2} as a \emph{heteroclinic orbit} $U$, with $[U(x_1)](x_2)=u(x_1,x_2)$, connecting in an appropriate space of functions $\mathcal H$ , the orbits $e^\pm$, which are regarded as the minima of an effective potential $\mathcal W$ defined in $\mathcal H$.

On the other hand, given a double well potential $W:\R^m\to[0,\infty)$, we consider in Theorems \ref{th1ground} and \ref{th2ground}, a heteroclinic orbit $e$ which is a local (but not global) minimum of $E_1$ in the affine space $e+H^1(\R;\R^m)$. Then, setting $y=(x_1,\ldots,x_d)$, we obtain a nontrivial solution $u:\R^{d+1}\to\R^m$ of system \eqref{system0} such that
 \begin{equation}\label{lay2}
\lim_{|y|\to\pm\infty}u(y,x_{d+1})=e(x_{d+1}).
\end{equation}
This is now a solution of \emph{homoclinic} type converging as $|y|\to\infty$ to the orbit $e$, which is a local minimum of the potential $\mathcal W$ defined in \eqref{defw}. Thus, Theorems \ref{th1ground} and \ref{th2ground} provide a new relevant example of layered solution, that can be seen as the homoclinic analog of the heteroclinic double layers solution.

These results also outline the hierarchical structure of solutions of system \eqref{system0}, since higher dimensional solutions such as those in \eqref{lay1} or \eqref{lay2}, reproduce the structure of lower dimensional solutions (e.g. the heteroclinic in \eqref{systode} or the ground state solution in \eqref{groundpb}) in an appropriate space of functions $\mathcal H$. Similarly, there is a perfect  analogy between the potential $W$ defined in $\R^m$, and the effective potential $\mathcal W$ defined in $\mathcal H$. This is clear if we compare the minimization problems \eqref{minpb} and \eqref{minpb22}: the local minimum $0$ of $W$ has its counterpart in the heteroclinic $e$, which is a local  minimum of $\mathcal W$, while the $d$-dimensional solution of \eqref{groundpb}, has its counterpart in the  the $d+1$-dimensional solution provided by Theorems \ref{th1ground} and \ref{th2ground}. For additional examples of such layered solutions, we refer for instance to the constructions in \cite{clerc,filament}.

Finally, we point out that provided that the solution $e$ in \eqref{solhet} is a local minimum of $E_k$ (or its renormalization), other applications of our method can be considered to construct the corresponding ground state solutions. We shall give some more examples in our subsequent work. In addition, by working in the affine space $e+H^1(\R^k;\R^m)$ (instead of $e+L^2(\R^k;\R^m)$), ground state solutions can be obtained for fourth order P.D.E. systems. For more details, we refer to \cite[Theorem 1.3]{ps2} and \cite{doublehet}.

\section{Existence of a ground state in Hilbert spaces}\label{sec:groundH}

In this section, we solve the minimization problem \eqref{minpb22} under appropriate conditions on the potential $\mathcal W$, which differ for $d \geq 3$ and $d=2$.

\subsection{The case of three or more dimensions in the domain}

\

Let $\mathcal H$ be a separable Hilbert space. We 
consider a potential $\mathcal W$ satisfying the following assumptions:
\begin{itemize}
 \item[(\textbf{H1})] $\mathcal W:\mathcal H\to [-\mathcal W_m,+\infty]$ is a sequentially weakly lower semicontinuous function such that  $\inf_{\mathcal H}\mathcal W=-\mathcal W_m\in(-\infty,0)$.
 \item[(\textbf{H2})] $\mathcal W(U)\geq 0$, if $\| U\|_{\mathcal H}\leq \delta$ (with $\delta>0$).  Consequently, it also holds $\mathcal W^-(U)\leq M\|U\|^{2^*}_{\mathcal H}$,  $\forall U\in\mathcal H$, with $2^*=\frac{2d}{d-2}$, $d\geq 3$, and a constant $M>0$.
 \item[(\textbf{H3})] $\mathcal W(0)=0$, and there exists $U_0$ such that $\mathcal W(U_0)<0$.
\end{itemize}
That is, the potential $\mathcal W$ which is bounded from below, has a local minimum equal to $0$ at $U=0$, and is negative at some point $U_0$.

When the domain $\R^d$ is of dimension $d\geq 3$, we define the space
$$\mathcal D^{1,2}_{rad}(\R^d;\mathcal H)=\{ U \in  L^{2^*}(\R^d;\mathcal H): \nabla U \in L^2(\R^d;\mathcal H^d), U=U\circ g, \forall g \in O(\R^d)\},$$
with $2^*=\frac{2d}{d-2}$ (cf. \cite{kreuter} for the general theory of Sobolev spaces of vector valued functions, and \cite[Chapter 1]{willem} for the definition and properties of $\mathcal D^{1,2}(\R^d)$),
as well as the class
\begin{equation}\label{class1}
\mathcal A=\left\lbrace U\in\mathcal D^{1,2}_{rad}(\R^d;\mathcal H) ,   \ \mathcal W(U)\in L^1(\R^d), \ \ \int_{\R^d}\mathcal W(U)\leq -1 \right\rbrace,
\end{equation}
which is nonempty if for instance we assume that
\begin{itemize}
 \item[(\textbf{H4})] $\sup_{t\in[0,1]} |\mathcal W(tU_0)|<\infty$.
\end{itemize}
We point out that the Sobolev inequality holds in the setting of Hilbert spaces. It is proved in \cite[Theorem 4.15]{kreuter} for $W^{1,2}(\R^d;\mathcal H)$, and it can be extended to $\mathcal D^{1,2}_{rad}(\R^d;\mathcal H)$ by using \cite[Remark 3]{brezis}. Thus, we have
$$\|U\|_{L^{2^*}(\R^d;\mathcal H)}\leq\|\nabla U\|_{L^2(\R^d;\mathcal H^d)}, \ \forall U\in\mathcal D^{1,2}_{rad}(\R^d;\mathcal H).$$
In the theorem below, we prove the existence of a minimizer for the problem
\begin{align}\label{minpb1}
T=\inf\left\lbrace \int_{\R^d}\dfrac{1}{2}\|\nabla U(y)\|_{\mathcal H^d}^2\,  dy:U\in\mathcal A \right\rbrace.
\end{align}

\begin{theorem}\label{th1}
Under assumptions (\textbf{H1})--(\textbf{H3}), if $\mathcal{A}\neq\emptyset$, there exists $V\in\mathcal{A}$, such that  $$\int_{\R^d}\dfrac{1}{2}\|\nabla V\|^2_{\mathcal{H}^d}=T, \  V \not \equiv 0,  \text{ and }  \int_{\R^d}\mathcal{W}(V)=-1.$$
\end{theorem}
\begin{proof} Let $U_j$ be a minimizing sequence i.e. $U_j \in \mathcal A $ is such that  $$\lim_{j\to\infty}\dfrac{1}{2}\int_{\R^d}\|\nabla U_j\|^2_{\mathcal H^d} =T.$$ It is clear that $\|U_j\|_{L^{2^*}(\R^d;\mathcal H)}$ and $\|\nabla U_j\|_{L^2(\R^d;\mathcal H^d)}$ are uniformly bounded. Since $\|U_j\|_{\mathcal H}$ is a radial scalar function belonging to $\mathcal D^{1,2}_{rad}(\R^d)$, and $|\nabla\|U_j\|_{\mathcal H}|\leq\|\nabla U_j\|_{\mathcal H^d}$, one can apply the radial lemma A.III  in \cite{BerLionsI}, to obtain the bound
\begin{equation*}\label{abound}
\|U_j(y)\|_{\mathcal H}\leq C_d \|\nabla U_j\|_{L^2(\R^d;\mathcal H^d)}|y|^{\frac{2-d}{2}}, \text{ for a.e. $y\neq 0$, where $C_d>0$, only depends on $d$}.
\end{equation*}
In particular
\begin{equation}\label{bbound}
\|U_j(y)\|_{\mathcal H}\leq C |y|^{\frac{2-d}{2}}, \text{ for a.e. $y\neq 0$, and for a constant $C>0$, uniformly in $j$}.
\end{equation}
 As a consequence, setting $\rho=(\frac{C}{\delta})^{\frac{2}{d-2}}$, we have $\|U_j(y)\|_{\mathcal H}\leq \delta$ for $|y|\geq \rho>0$, uniformly in $j$.

From the previous uniform bounds, we deduce that up to subsequence:
\begin{itemize}
\item $\nabla U_j \rightharpoonup \nabla V$ in $L^2(\R^d;\mathcal H^d)$, and $\|\nabla V\|^2_{L^2(\R^d;\mathcal H^d)}\leq\liminf_{j\to\infty}\|\nabla U_j\|^2_{L^2(\R^d;\mathcal H^d)}=T$.
\item $U_j\rightharpoonup V$ in $L^{2^*}(\R^d;\mathcal H)$, and $V$ is radial.
\item $U_j(y) \rightharpoonup V(y)$, for a.e. $ y \in \R^d$.
\end{itemize}
The almost everywhere weak convergence $U_j(y) \rightharpoonup V(y)$, follows from the fact that the closed balls $B_R=\{U\in \mathcal H: \|U\|_{\mathcal H}\leq R\}$ are compact and metrizable in the weak topology, with a metric $d$ satisfying $d(U,\tilde U)\leq R\|U-\tilde U\|_{\mathcal H}$ (cf. for instance \cite[Theorem 3.29]{brez}). Moreover, the radial symmetry of $U_j$ and the bound \eqref{bbound} imply that the maps $U_j(r)$, viewed as functions of $r=|y|$, are uniformly bounded in $H^{1}([a,b]);\mathcal H)$, for every $0<a<b<\infty$.  Consequently, the sequence $U_j(r)$ is equibounded and equicontinuous on $[a,b]$, both with respect to the norm $\|\cdot\|_{\mathcal H}$, and to the distance $d$ induced by the weak topology. In view of the theorem of  Ascoli applied to the sequence $U_j(r)$ via a diagonal argument, we deduce that
$U_j(r) \rightharpoonup U(r)$, for every $ r>0$, and for a function $U$ satisfying  $\|U(r)\|_{\mathcal H}\leq C r^{\frac{2-d}{2}}$. Finally, given $r_0>0$, we have
$U_j(r)-U_j(r_0)=\int_{r_0}^r U'_j(s)ds$, and letting $j\to\infty$ we get $U(r)-U(r_0)=\int_{r_0}^r V'(s)ds$, $\forall r>0$. This proves that $U'=V'$. To conclude that $U=V$, we notice that $\lim_{r\to \infty}U(r)=\lim_{r\to\infty}V(r)=0$, since also $V$ satisfies \eqref{bbound}.

At this stage, one can see that the inequality $ \mathcal W^-(U_j)\leq M \|U_j\|^{2^*}_{\mathcal H}$, implies that the integrals
$\int_{\R^d }\mathcal W^-(U_j) $ are uniformly bounded. On the other hand, since  $\int_{\R^d} \mathcal W^+(U_j)=\int_{\R^d} \mathcal W^-(U_j)+ \int_{\R^d} \mathcal W(U_j) $, and $\int_{\R^d}\mathcal W(U_j)\leq -1$, we conclude that $\int_{\R^d}\mathcal W^+(U_j)$ as well as $\int_{\R^d}|\mathcal W(U_j)|$ are uniformly bounded.

Finally, using Fatou's lemma and the sequentially weakly lower semicontinuity of $\mathcal W$, we obtain
$$\int_{\R^d\setminus B_\rho}\mathcal W(V) \leq \int_{\R^d\setminus B_\rho}\liminf_{j\to\infty}\mathcal W(U_j)\leq \liminf_{j\to\infty}\int_{\R^d\setminus B_\rho}\mathcal W(U_j),$$
and in particular $\int_{\R^d\setminus B_\rho}|\mathcal W(V(y))|dy<\infty$, since $\|V(y)\|_{\mathcal H}\leq\delta$, for $|y|\geq\rho$.
Moreover, in the ball $B_\rho$, we have again by Fatou's lemma
$$\int_{ B_\rho}(\mathcal W(V)+\mathcal W_m)\leq \int_{ B_\rho}\liminf_{j\to\infty}(\mathcal W(U_j)+\mathcal W_m)\leq \liminf_{j\to\infty}\int_{ B_\rho}(\mathcal W(U_j)+\mathcal W_m),$$
that is, $$\int_{ B_\rho}\mathcal W(V) \leq \liminf_{j\to\infty}\int_{ B_\rho}\mathcal W(U_j),$$ and in particular
 $\int_{B_\rho}|\mathcal W(V)|<\infty$. This proves that $\mathcal W(V)\in L^1(\R^d)$, and
 $$\int_{ \R^d}\mathcal W(V) \leq \liminf_{j\to\infty}\int_{\R^d}\mathcal W(U_j)\leq -1.$$
So far we established that $V\in\mathcal A$, and $T=\dfrac{1}{2}\int_{\R^d}\|\nabla V\|_{\mathcal H^d}^2$. Arguing as in \cite[Lemma 2.3]{brezis}, one can easily see by scaling (i.e. $U(y)\to U(\lambda y)$) that
\begin{equation}\label{eqsca}
\frac{1}{2}\int_{\R^d}\|\nabla U\|^2_{\mathcal H^d}\geq T \Big(-\int_{\R^d}\mathcal W(U)\Big)^{\frac{d-2}{d}},
\end{equation}
holds for every $U \in \mathcal D^{1,2}_{rad}(\R^d;\mathcal H)$, such that $\mathcal W(U)\in L^1(\R^d)$, and $\int_{\R^d}\mathcal W(U)<0$. Therefore, the minimizer $V$ satisfies $\int_{\R^d}\mathcal W(V)=-1$, and in particular $V\not\equiv 0$. \qed

\end{proof}

\subsection{The two dimensional case}

\

Let $\mathcal H$ be a separable Hilbert space. We consider a potential $\mathcal W$ satisfying the following assumptions:
\begin{itemize}
 \item[(\textbf{H1})] $\mathcal W:\mathcal H\to [-\mathcal W_m,+\infty]$ is a sequentially weakly lower semicontinuous function such that  $\inf_{\mathcal H}\mathcal W=-\mathcal W_m\in(-\infty,0)$.
 \item[(\textbf{H2})] $\mathcal W(U)\geq \alpha\|U\|^2_{\mathcal H}$, for $\|U\|_{\mathcal H}\leq \delta$, and for constants $\alpha,\delta>0$.   Consequently, it also holds $\mathcal W^-(U)\leq M\|U\|^{2}_{\mathcal H}$,  $\forall U\in\mathcal H$,  and for a constant $M>0$.
 \item[(\textbf{H3})] $\mathcal W(0)=0$, and there exists $U_0$ such that $\mathcal W(U_0)<0$.
\end{itemize}

In dimension two, we work with the space
$$ H^{1}_{rad}(\R^2;\mathcal H)=\{ U \in H^{1}(\R^2;\mathcal H) : U=U\circ g, \forall g \in O(\R^2)\},$$
and the class
\begin{equation}\label{class2}
\mathcal A=\left\lbrace U\in H^{1}_{rad}(\R^2;\mathcal H) :   U\neq 0, \ \mathcal W(U)\in L^1(\R^2), \ \ \int_{\R^2}\mathcal W(U)\leq 0 \right\rbrace,
\end{equation}
which is nonempty if for instance assumption (\textbf{H4}) holds.

We study the following minimization problem:
\begin{equation}\label{minpb2}
T=\inf\left\lbrace \int_{\R^2}\dfrac{1}{2}\|\nabla U(y)\|^2_{\mathcal H^2}\, dy : U\in\mathcal A \right\rbrace
\end{equation}

\begin{theorem}\label{th2}
Under assumptions (\textbf{H1})--(\textbf{H3}), if $\mathcal A\neq\emptyset$, there exists $V\in\mathcal A$, such that  $$\int_{\R^2}\dfrac{1}{2}\|\nabla V\|^2_{\mathcal H^2}=T,  \text{ and }  \int_{\R^2}\mathcal W(V)=0.$$
\end{theorem}
\begin{proof} Let $U_j$ be a minimizing sequence i.e. $U_j \in \mathcal A $ is such that  $$\lim_{j\to\infty}\dfrac{1}{2}\int_{\R^d}\|\nabla U_j\|^2_{\mathcal H^2} =T.$$ It is clear that $\|\nabla U_j\|_{L^2(\R^2;\mathcal H^2)}$ is uniformly bounded, and that $\int_{\|U_j\|_{\mathcal H}\leq\delta}\|U_j\|^2_{\mathcal H}>0$, since $\|U_j\|_{\mathcal H} \in H^1(\R^2)\setminus\{0\}$.
By rescaling, we may assume that $\int_{\|U_j\|_{\mathcal H}\leq\delta}\|U_j\|^2_{\mathcal H}=1$.

We first prove the following version of \cite[Radial lemma A.II]{BerLionsI} that will be applied to the radial scalar functions $\|U_j\|_{\mathcal H}\in H^1(\R^2)$.
\begin{lemma}
Every radial function $u \in H^{1}(\R^2)$ is almost
everywhere equal to a function $u(r)$ of the radius $r=|x|$, continuous for $r>0$, and such that
\begin{equation}
|u(r)|\leq \frac{1}{ \sqrt{2\pi r}} \Big(\int_{\{y\in\R^2: |y|>r\} } (|u(y)|^2+|\nabla u(y)|^2) dy\Big)^{\frac{1}{2}}, \ \forall r>0.
\end{equation}
\end{lemma}
\begin{proof}
By a standard density argument, it suffices to consider the case $u\in\mathcal D(\R^2)$. We have
$$\frac{d}{dr}(ru^2)=u^2+
2r u u'\geq -r(|u|^2+|u'|^2)\Rightarrow r u^2(r)\leq \int_r^\infty s(|u(s)|^2+|u'(s)|^2)ds.$$
Thus, $r u^2(r)\leq \frac{1}{2\pi}\int_{|y|>r} (|u(y)|^2+|\nabla u(y)|^2) dy$. \qed
\end{proof}
Next, we notice that assumption (\textbf{H2}) and the inequality $\int_{\R^2}\mathcal W(U_j)\leq 0$, imply that $U_j(\R^2)$ is not contained in the ball $B(0,\delta)$ of $\mathcal H$. Thus, for every $j$, we can define $$r_j=\max\{r>0: \|U_j(r)\|_{\mathcal H}\geq \delta\}.$$
Consequently, we have for a.e. $y\in\R^2$ such that $|y|\geq r_j$ :
$$\|U_j(y)\|_{\mathcal H}\leq  \frac{1}{ \sqrt{2\pi|y|}} \Big(\int_{\|U_j\|_{\mathcal H}\leq\delta} (\|U_j\|_{\mathcal H}^2+\|\nabla U_j\|^2_{\mathcal H^2}) \Big)^{\frac{1}{2}}\leq  \frac{\sqrt{T+1}}{ \sqrt{\pi|y|}}, $$
provided that $j$ is large enough. In particular, we obtain that $r_j  \leq  \frac{T+1}{ \pi \delta^2}$, hence the inequality $\|U_j(y)\|_{\mathcal H}\geq \delta$ may occur only for $|y| \leq \rho:= \frac{T+1}{ \pi \delta^2}$. Finally, since $(\|U_j\|_{\mathcal H}-\delta)^+$ vanishes for
 $|y| \geq  \frac{T+1}{ \pi \delta^2}$, it follows from the Poincar\'{e} inequality that $(\|U_j\|_{\mathcal H}-\delta)^+$ is uniformly bounded in $L^2(\R^2)$. Therefore, we can see that $U_j$ is also uniformly bounded in $L^2(\R^2;\mathcal H)$.

From the previous uniform bounds, we deduce that up to subsequence:
\begin{itemize}
\item $\nabla U_j \rightharpoonup \nabla V$ in $L^2(\R^2,\mathcal H^2)$, and $\|\nabla V\|^2_{L^2(\R^2,\mathcal H^2)}\leq\liminf_{j\to\infty}\|\nabla U_j\|^2_{L^2(\R^2,\mathcal H^2)}=T$.
\item $U_j\rightharpoonup V$ in $L^{2}(\R^2,\mathcal H)$, and $V$ is radial.
\item $V_j(y) \rightharpoonup V(y)$, for a.e. $ y \in \R^2$.
\end{itemize}
The almost everywhere weak convergence $U_j(y) \rightharpoonup V(y)$, is proved exactly as in the higher dimensional case.

Next, one can see (as in the higher dimensional case) that the integrals $\int_{\R^2}|\mathcal W(U_j)|$ are uniformly bounded. Similarly, by using Fatou's lemma and the sequentially weakly lower semicontinuity of $\mathcal W$, we get $\mathcal W(V)\in L^1(\R^2)$, and $\int_{ \R^2}\mathcal W(V) \leq 0$.

It remains to show that $V\neq 0$. To see this, we consider the modified potential
\begin{equation*}
 \mathcal{\widetilde W}(U)=\begin{cases}
0 &\text{ when } \|U\|_{\mathcal H}\leq \delta\\
\mathcal W(U) &\text{ when } \|U\|_{\mathcal H}> \delta,
\end{cases}
\end{equation*}
which is sequentially weakly lower semicontinuous. By applying as previously Fatou's lemma to the sequence $\mathcal{\widetilde W}(U_j)$, it follows that
$$\int_{ \R^2}\mathcal{\widetilde W}(V) \leq \liminf_{j\to\infty}\int_{\R^2} \mathcal{\widetilde W}(U_j)\leq  \liminf_{j\to\infty}\Big(\int_{\R^2} \mathcal{ W}(U_j)-\alpha\int_{\|U_j\|_{\mathcal H}\leq \delta}\|U_j\|^2_{\mathcal H} \Big)\leq -\alpha,$$
which proves that $V\neq 0$.

Finally, we establish that $\int_{\R^2}\mathcal W(V)=0$. Let us assume by contradiction that $\int_{\R^2}\mathcal W(V)<0$. Then, we proceed as in  \cite[Lemma 4.1.]{smy2} to derive a Pohozaev type identity for $V$. For every $0<r<s$ fixed, and $\kappa\in (0,\frac{s}{r})$, we define the radial comparison map
\begin{equation}
V_\kappa(t)=\begin{cases}
V(\frac{t}{\kappa}) &\text{ when } 0\leq t\leq \kappa r,\\
V(r+(s-r)\frac{t-\kappa r}{s-\kappa r}) &\text{ when } \kappa r\leq t\leq s,\\
V(t)&\text{ when } t\geq s.
\end{cases}
\end{equation}
Here, by abuse of notation, we write $V_\kappa(t)$ instead of $V_\kappa(y)$, where $t = |y|$, $y\in\R^2$.

One can check that when $\kappa$ is sufficiently close to $1$, we have
\begin{multline*}
\int_{\R^2}\mathcal W(V_\kappa)=\kappa^2 \int_{B(0,r)}\mathcal W(V)+ \int_{\R^2 \setminus B(0,s)}\mathcal W(V)\\
+2\pi\int_r^s \big((s-\kappa r)\frac{t-r}{s-r}+\kappa r\big)\frac{s-\kappa r}{s-r} \mathcal W(V(t))d t \leq 0,
\end{multline*}
which implies that $V_\kappa \in\mathcal A$. Therefore, setting $f(\kappa):=\int_{\R^2}\frac{1}{2}\|\nabla V_\kappa\|_{\mathcal H^2}^2$, we deduce that
\begin{align*}
f(\kappa)&=\int_0^r\pi t\|V'(t)\|_{\mathcal H}^2 dt+\int_r^s\pi\big((s-\kappa r)\frac{t-r}{s-r}+\kappa r\big)\frac{s-r}{s-\kappa r} \|V'(t)\|_{\mathcal H}^2d t,
\end{align*}
and $f'(1)=\frac{\pi r s}{s-r}\int_r ^s\|V'(t)\|^2_{\mathcal H}=0$, in view of the minimality of $V$. This implies that $V'\equiv 0$, and $V\equiv 0$, which is a contradiction. \qed
\end{proof}

\section{The ground state corresponding to a heteroclinic orbit}\label{sec:application}

In this section, we consider the elliptic system
\begin{equation}\label{system}
\Delta u(x_1,\ldots,x_{d+1})= \nabla W( u(x_1,\ldots,x_{d+1})), \ u:\R^{d+1}\to\R^m,  \ d,m\geq 2,
\end{equation}
where the potential $W$ is such that
\begin{subequations}\label{w13}
\begin{equation}\label{w1}
\text{$W\in C^{2}(\R^m; \R)$ is nonnegative, and has (at least) two zeros $a^-$ and $a^+$,}
\end{equation}
\begin{equation}\label{w2}
\text{the Hessian matrices $D^2 W(a^\pm)$ are positive definite,}
\end{equation}
\begin{equation}\label{w3b}
\text{the map $u\mapsto \nabla W(u)$ is globally Lipschitz in $\R^m$.}
\end{equation}
Moreover, we assume the existence of a heteroclinic orbit $e\in C^2(\R;\R^m)$ connecting $a^-$ to $a^+$:
\begin{equation}\label{ode0}
e''(t)=\nabla W(e(t)), \ \lim_{t\to\pm\infty}e(t)=a^\pm.
\end{equation}
\end{subequations}
Since the minima $a^\pm$ are nondegenerate, we know that $e$ converges exponentially fast to the minima $a^\pm$ (cf. for instance \cite{antonop}). In particular, we have $W(e),|e'|^2, |e''|^2\in L^1(\R)$. We also recall that the Action functional associated to system \eqref{ode0} is denoted by
\begin{equation}\label{energy1}
E_1(v)=\int_{\R}\big(\frac{1}{2}|v'|^2+W(v)\big), \text{ or } E_1(v,I)=\int_{I}\big(\frac{1}{2}|v'|^2+W(v)\big),  \ I\subset \R.
\end{equation}
In order to construct a ground state corresponding to the heteroclinic orbit $e$, we shall work in the affine space $e+L^2(\R;\R^m)$, and consider the following renormalized functional:
\begin{equation}
\mathcal W(h)=
\begin{cases}
E_1(e+h)-E_1(e), &\text{ when the distributional derivative }  h' \in L^2(\R;\R^m),\\
+\infty,&\text{ otherwise,}
\end{cases}
\end{equation}
defining a potential $\mathcal W$ in the Hilbert space $\mathcal H=L^2(\R;\R^m)$. It is clear that $\mathcal W$ takes its values in $[-\mathcal W_m,+\infty]$, with $\mathcal W_ m=E_1(e)$, and one can see as in \cite[Lemma 3.1. (i)]{ps2} that $\mathcal W$ is sequentially weakly lower semicontinuous. In addition, when $h \in H^1(\R;\R^m)$, an integration by parts gives
$\int_\R \nabla W(e)\cdot h=\int_\R e''\cdot h=-\int_\R e'\cdot h'$, since $e'\in H^{1}(\R;\R^m)$. Thus, we derive an alternative expression of the potential $\mathcal W$:
\begin{align}\label{alter}
\mathcal W(h)&=\int_\R \Big[\frac{1}{2}|e'+h'|^2-\frac{1}{2}|e'|^2 +W(e+h)-W(e)\Big]\nonumber \\
&=\int_\R \Big[\frac{1}{2}|h'|^2  +W(e+h)-W(e)- \nabla W(e)\cdot h\Big].
\end{align}
Setting $f(t,h)=W(e(t)+h)-W(e(t))- \nabla W(e(t))\cdot h$, we have in view of \eqref{w3b}:
\begin{equation}\label{abc}
|f(t,h)|\leq L|h|^2, \text{ $\forall t\in\R$, $\forall h\in \R^m$, and for a constant $L>0$.}
\end{equation}

 \subsection{Ground state for $d\geq 3$} When $d\geq 3$, we also assume that
\begin{subequations}\label{ass2}
\begin{equation}\label{nondegw}
\text{$\mathcal W(h)\geq 0$ holds for $\|h\|_{\mathcal H}\leq \delta$ (with $\delta>0$ small enough),}
\end{equation}
\begin{equation}\label{negw}
\text{and $\mathcal W(h_0)<0$, for some $h_0\in H^1(\R;\R^m)$. }
\end{equation}
\end{subequations}
In Proposition \ref{propo1}, we provide examples of potentials $W$ satisfying assumptions \eqref{w13} and \eqref{ass2}.
Provided that \eqref{w13} as well as \eqref{ass2} hold (with $d\geq 3$), it is clear that the potential $\mathcal W$ satisfies the assumptions (\textbf{H1})-(\textbf{H4}) of Theorem \ref{th1}. Therefore, there exists a radial map $V\in\mathcal A$ solving the minimization problem \eqref{minpb1}, in the class
\begin{equation}\label{class1bis}
\mathcal A=\left\lbrace U\in\mathcal D^{1,2}_{rad}(\R^d;\mathcal H) ,   \ \mathcal W(U)\in L^1(\R^d), \ \ \int_{\R^d}\mathcal W(U)\leq -1 \right\rbrace,
\end{equation}
Moreover, we have $\int_{\R^d}\mathcal W(V)=-1$, and $V\not\equiv 0$.

Setting $x=(x_1,\ldots, x_{d+1})=(y,x_{d+1})$, with $y=(x_1,\ldots,x_d)$, we recover from the minimizer $V:\R^d\to\mathcal H=L^2(\R;\R^m)$ a solution $u$ to system \eqref{system}. Indeed, we shall prove that the map $v(x):=[V(y)](x_{d+1})$ solves the equation
\begin{equation}\label{prel}
\mu \Delta_y v(y,x_{d+1})+\frac{\partial^2 v}{\partial x_{d+1}^2}(y,x_{d+1})=\nabla W(e(x_{d+1})+v(y,x_{d+1}))-\nabla W(e(x_{d+1})),
\end{equation}
for $(y,x_{d+1})\in \R^d\times\R$, and for a Lagrange multiplier $\mu>0$. Then, setting $\tilde v(x)=v(\mu^{-\frac{1}{2}}y,x_{d+1})$, it follows that the map $u(x)=\tilde v(x)+e(x_{d+1})$ is a solution to system \eqref{system} satisfying the asymptotic properties of Theorem \ref{th1ground} below.

Finally, this solution $u$ can be characterized variationally, by introducing the following renormalized energy functional for $\tilde u(x)=e(x_{d+1})+\tilde \xi(x)$, with $\tilde \xi(x)=[\tilde\Xi(y)](x_{d+1})$, $\tilde\Xi\in\mathcal D^{1,2}_{rad}(\R^d;\mathcal H)$, and $\mathcal W(\tilde\Xi)\in L^1(\R^d)$:
\begin{align}\label{renormene}
\mathcal E(\tilde u)
=\int_{\R^{d}}\big(\frac{1}{2}\|\nabla \tilde \Xi(y)\|^2_{\mathcal H^d}+\mathcal W(\tilde\Xi(y))\big)dy.
\end{align}
Notice that $\mathcal E$ is related to the functional
\begin{equation}\label{energybis}
E_{d+1}(v,\Omega)=\int_{\Omega}\big(\frac{1}{2}|\nabla v|^2+W(v)\big), \ \Omega\subset\R^{d+1},
\end{equation}
associated to system \eqref{system}, since given a test function $\phi\in\mathcal D(\R^{d+1};\R^m)$, one can check that
\begin{multline}
E_{d+1}(e(x_{d+1})+\phi(x),\operatorname{supp}\phi)-E_{d+1}(e(x_{d+1}),\operatorname{supp}\phi)\\=\int_{\R^{d+1}}\big(\frac{1}{2}|\nabla\phi(x)|^2+f(x_{d+1},\phi(x))\big)dx =\mathcal E(e(x_{d+1})+\phi(x)).
\end{multline}
\begin{theorem}\label{th1ground}
Let $d\geq 3$. Under assumptions \eqref{w13} and \eqref{ass2}, there exists a  solution $u\in C^2(\R^d\times\R;\R^m)$ to system \eqref{system} satisfying the following properties.
\begin{itemize}
\item[(i)]  $u(y,x_{d+1})=u(z,x_{d+1})$, whenever $|y|=|z|$ (with $,y,z\in \R^d$).
\item[(ii)] $\lim_{|x|\to\infty}(u(x)-e(x_{d+1}))=0$ (with $x=(x_1,\ldots, x_{d+1})$), and $u(x)-e(x_{d+1})\not\equiv 0$. In particular $\lim_{y\in\R^d, |y|\to\infty}u(y,x_{d+1})=e(x_{d+1})$.
\item[(iii)] Let $\tilde u$ be a weak solution to system \eqref{system} such that $\tilde u(x)=e(x_{d+1})+\tilde \xi(x)$, with  $\tilde\xi(x)=[\tilde\Xi(y)](x_{d+1})$, $\tilde\Xi\in\mathcal D^{1,2}_{rad}(\R^d;\mathcal H)$, $\tilde\Xi \not\equiv 0$, and $\mathcal W(\tilde\Xi)\in L^1(\R^d)$. Then, we have $\mathcal E(u)\leq\mathcal E( \tilde u)$.
\end{itemize}
\end{theorem}
\begin{proof}

Let $V\in\mathcal A$, be the minimizer provided by Theorem \ref{th1}. We proceed in several steps to recover the solution $u$  from the minimizer $V$. In what follows $B_d(0,R)$ is a ball of $\R^d$ centered at $0$.

\begin{lemma}\label{recover}
Let $v(x):=[V(y)](x_{d+1})$, then, we have $v\in H^1(B_d(0,R)\times\R;\R^m)$. On the other hand, let $H^1_{0,rad}(B_d(0,R)\times\R;\R^m)$ be the subspace of functions $\phi \in  H^1_0(B_d(0,R)\times\R;\R^m)$ which are radial with respect to $y$, i.e. $ \phi( gy,x_{d+1})=\phi(y,x_{d+1})$, $\forall g \in O(\R^d)$, and for a.e. $(y,x_{d+1})\in B_d(0,R)\times\R$. Any $\phi\in  H^1_{0,rad}(B_d(0,R)\times\R;\R^m)$, is extended to a map of $H^1(\R^{d+1};\R^m)$ by setting $\phi\equiv 0$ on $\R^{d+1}\setminus B_d(0,R)\times\R$. Then, setting $[\Phi(y)](x_{d+1}):=\phi(x)$, we have $V+\lambda\Phi\in\mathcal D^{1,2}_{rad}(\R^d;\mathcal H)$, and $\mathcal W(V+\lambda\Phi)\in L^1(\R^d)$, for every $\lambda\in\R$.
\end{lemma}
\begin{proof}
Since $V\in L^{2^*}(\R^d;\mathcal H)$, we also have that $V\in L^{2}(B_d(0,R);\mathcal H)$, for every ball $B_d(0,R)$ of $\R^d$. In view of the isomorphism
$$L^2(B_d(0,R)\times \R;\R^m)\simeq L^2(B_d(0,R);L^2(\R;\R^m)),$$
we deduce that $v \in L^2(B_d(0,R)\times \R;\R^m)$. Similarly, since $ V_{x_1},\ldots,V_{x_d}$ belong to $L^2(\R^d;L^2(\R;\R^m))$, it follows that the weak derivatives $v_{x_1},\ldots,v_{x_d}$ exist and belong to $L^2(\R^{d+1};\R^m)$. Next, in view of \eqref{alter} and \eqref{abc}, we have
\begin{align*}
\int_{B_d(0,R)\times\R}\frac{1}{2}\Big|\frac{d [V(y)]}{d x_{d+1}}(x_{d+1})\Big|^2dx &=\int_{B_d(0,R)}\mathcal W(V)- \int_{B_d(0,R)\times\R}f(x_{d+1},v(x))dx<\infty.
\end{align*}
This implies that the derivative $v_{x_{d+1}}$ exists and belongs to $L^2(B_d(0,R)\times\R;\R^m)$.

On the other hand, given $\phi\in  H^1_{0,rad}(B_d(0,R)\times\R;\R^m)$, the corresponding map $\Phi$ is radial and belongs to  $H^1_0(B_d(0,R);L^2(\R;\R^m))$ as well as to $H^1(\R^d;L^2(\R;\R^m))$ which is embedded in $L^{2^*}(\R^d;L^2(\R;\R^m))$ (cf. \cite[Theorem 4.15]{kreuter}).
This proves that $V+\lambda\Phi\in\mathcal D^{1,2}_{rad}(\R^d;\mathcal H)$. Moreover, since
$$\int_{B_d(0,R)}|\mathcal W(V+\lambda\Phi)|\leq\int_{B_d(0,R)\times\R}\Big(\frac{1}{2}|(v_{x_{d+1}}+\lambda   \phi_{x_{d+1}})(x)|^2+|f(x_{d+1},(v+\lambda\phi)(x))|\Big)dx,$$
it is clear that $\mathcal W(V+\lambda\Phi)\in L^1(\R^d)$.\qed
\end{proof}

In view of Lemma \ref{recover}, given $R>0$, $v$ is a minimizer of the functional $$S(w)=\int_{B_d(0,R)\times \R}|\nabla_y w(x)|^2 dx,$$ in the class $$\mathcal C=\{w=v+\phi: \phi \in  H^1_{0,rad}(B_d(0,R)\times\R;\R^m): F(w)\leq -1\},$$
where $$F(w)=\int_{B_d(0,R)\times\R} \Big(\frac{1}{2}|w_{x_{d+1}}(x)|^2+f(x_{d+1},w(x))\Big)dx=\int_{B_d(0,R)} \mathcal W(V+\Phi)dy.$$
Next, we define the continuous linear functionals on $H^1_{0,rad}(B_d(0,R)\times\R;\R^m)$:
$$G_1(\phi)=\int_{\R^{d+1}}\big(v_{x_{d+1}}(x)\cdot\phi_{x_{d+1}}(x)+\big[\nabla W(e(x_{d+1})+v(x))-\nabla W(e(x_{d+1}))\big]\cdot\phi(x)\big) dx,$$
and $$G_2(\phi)=\int_{\R^{d+1}}(v_{x_1}\cdot \phi_{x_1}+\ldots  +v_{x_d}\cdot \phi_{x_d} )dx,$$
where $\phi\in H^1_{0,rad}(B_d(0,R)\times\R;\R^m)$ is extended to a map of $H^1(\R^{d+1};\R^m)$ by setting $\phi\equiv 0$, on $ \R^{d+1}\setminus B_d(0,R)\times\R$. One can see that
$$\frac{d}{d\lambda}F(v+\lambda\phi)\Big|_{\lambda=0}=G_1(\phi).$$

In the next lemma we show that minimizing functional $S$ under the constraint $F\leq -1$, provides a Lagrange multiplier for the functionals $G_1$ and $G_2$.
\begin{lemma}\label{recover2} Let $R>0$ be large enough so that $\int_{B_d(0,R)}\|\nabla V(y)\|^2_{\mathcal H^d}dy>0$. Then, there exists a constant $\mu>0$ (independent of $R$) such  that $\mu G_2+G_1=0$.
\end{lemma}
\begin{proof}
Our first claim is that
\begin{equation}\label{claim Lag}
G_1(\phi)=0\Longrightarrow G_2(\phi)=0.
\end{equation}
 Indeed, if $G_1(\phi)=0$, then we get $$\int_{\R^d}\mathcal W(V+\lambda\Phi)dy=\int_{\R^d}\mathcal W(V(y))dy +\lambda\epsilon(\lambda)=-1+\lambda\epsilon(\lambda),$$
with $\lim_{\lambda\to 0}\epsilon(\lambda)=0$. In particular, the map $\Psi_\lambda(y)=(V+\lambda \Phi)(|-1+\lambda \epsilon(\lambda)|^{\frac{1}{d}}y)$ is such that $\mathcal W(\Psi_\lambda)=-1$. Thus, in view of the minimality  of $V$, we have for $\lambda$ close to $0$:
$$\int_{\R^d}\|\nabla V(y)\|^2_{\mathcal H^d} dy\leq \int_{\R^d}\|\nabla \Psi_\lambda(y)\|^2_{\mathcal H^d} dy=(1-\lambda \epsilon(\lambda))^{\frac{2-d}{d}} \int_{\R^d}\|\nabla (V+\lambda\Phi)(y))\|^2_{\mathcal H^d} dy,$$
which implies that $\int_{\R^d}\langle\nabla V(y)\cdot\nabla \Phi(y)\rangle_{\mathcal H^d} dy=G_2(\phi)=0$.

Next, we check that $G_2\not\equiv 0$, provided that $\int_{B_d(0,R)}\|\nabla V\|^2_{\mathcal H^d}>0$. Indeed, if $G_2 \equiv 0$, the map $v$ which is radial with respect to $y$ would satisfy $\Delta_y v(y,x_{d+1})=0$ in $B_d(0,R)\times\R$ (cf. the principle of symmetric criticality in \cite{palais}). Hence, we would obtain that $v(y,x_{d+1})=A(x_{d+1})+B(x_{d+1})|y|^{2-d}$, for some functions $A,B:\R\to\R$. However, the integrability of $|\nabla_y v|^2$ on $B_d(0,R)\times\R$, implies that $B\equiv 0$, and $\nabla_y v \equiv 0$, which is a contradiction.

In view of \eqref{claim Lag}, it follows that $G_1\not\equiv 0$, and both linear functionals $G_1$ and $G_2$ are proportional. Let $\mu\neq 0$, be such that $\mu G_2+G_1=0$. By choosing a test function $\phi$ such that $G_1(\phi)\neq 0$, one can see that $\mu=-\frac{G_1(\phi)}{G_2(\phi)}$ is independent of $R$ (large enough). Moreover, if $G_1(\phi)<0$, then $V+\lambda\Phi\in \mathcal A$ provided that $\lambda>0$ is small enough. Thus, the minimality of $V$ implies that $G_2(\phi)\geq 0$. This proves that $\mu>0$. \qed
\end{proof}
So far we established that
\begin{multline}
0=\int_{\R^{d+1}}\big(\mu(v_{x_1}\cdot \phi_{x_1}+\ldots  +v_{x_d}\cdot \phi_{x_d} )+v_{x_{d+1}}\cdot\phi_{x_{d+1}}\big)\\
+\int_{\R^{d+1}}(\nabla W(e(x_{d+1})+v(x))-\nabla W(e(x_{d+1}))\cdot\phi(x)\big) dx, \ \forall \phi\in\mathcal D_{rad}(\R^{d+1};\R^m), \end{multline}
where $\mathcal D_{rad}(\R^{d+1};\R^m)$ is the subspace of test functions which are radial with respect to $y$. Setting $\tilde v(x)=v(\mu^{-\frac{1}{2}}y,x_{d+1})$, we deduce in view of the principle of symmetric criticality (cf. \cite{palais}) that
$$\Delta \tilde v(x)=\nabla W(e(x_{d+1})+\tilde v(x))-\nabla W(e(x_{d+1})) \text{ holds weakly in } \R^{d+1}.$$
Therefore, $u(x)=\tilde v(x)+e(x_{d+1})$ is a weak solution of system \eqref{system}. In the next lemma, we study its asymptotic properties.

\begin{lemma}\label{AS}(Asymptotic behaviour)
We have
\begin{itemize}
\item[(i)] $\lim_{|x|\to\infty}\|\tilde v\|_{H^1(B(x,1);\R^m)}=0$.
 \item[(ii)] $\tilde v \in C^2(\R^{d+1};\R^m)$, and $\lim_{|x|\to\infty}\tilde v(x)=0$, as well as $\lim_{|x|\to\infty}\nabla \tilde v(x)=0$.
 \end{itemize}
\end{lemma}
\begin{proof}
(i) Let $x_n=(y_n,z_n)\in\R^d\times\R$ be a sequence such that $\lim_{n\to\infty}|x_n|=\infty$. If $\lim_{n\to\infty}|y_n|=\infty$, then it is clear that
$\lim_{n\to\infty}\|\tilde v\|_{L^2(B(x,1):\R^m)}=0$, since  $\lim_{|y|\to\infty}\|V(y)\|_{L^2(\R;\R^m)}=0$. Otherwise, if $\sup_n| y_n|\leq R$, for a constant $R$, then we have $\lim_{n\to\infty}|z_n|=\infty$, and we know that $\tilde v \in L^2(B(0,R+1)\times \R;\R^m)$. Consequently, we also obtain $\lim_{|x|\to\infty}\|\tilde v\|_{L^2(B(x,1);\R^m)}=0$. Finally, it is clear that for $i=1,\ldots,d$, we have $\lim_{|x|\to\infty}\| \tilde v_{x_i}\|_{L^2(B(x,1);\R^{m})}=0$, since  $ \tilde v_{x_i}\in L^2(\R^{d+1};\R^{m})$. On the other hand,  in view of the inequality
\begin{align*}
\int_{B_d(y_n,R)\times\R}\frac{1}{2}|v_{x_{d+1}}|^2dx &=\int_{B_d(y_n,R)}\mathcal W(V)- \int_{B_d(y_n,R)\times\R}f(x_{d+1},v(x))dx\\
&\leq \int_{B_d(y_n,R)}\mathcal W(V)+L\|v\|^2_{L^2(B_d(y_n,R)\times\R;\R^m)},
\end{align*}
one can check as previously that  $\lim_{|x|\to\infty}\| \tilde v_{x_{d+1}}\|_{L^2(B(x,1);\R^{m})}=0$.

(ii) Set $g(x)=\nabla W(e(x_{d+1})+\tilde v(x))-\nabla W(e(x_{d+1}))$. In view of i) and since the map $\R^{m}\ni u\mapsto\nabla W(u)$ is globally Lipschitz, it is clear that $g\in H^1_{loc}(\R^{d+1};\R^m)$, and  $\lim_{|x|\to\infty}\|g\|_{H^1(B(x,1);\R^m)}=0$. Therefore, we obtain by elliptic regularity that $\tilde v\in H^3_{loc}(\R^{d+1};\R^m)$, and  $\lim_{|x|\to\infty}\|\tilde v\|_{H^3(B(x,1/2);\R^m)}=0$. Next, by a standard bootstrap argument, it follows that  $\tilde v \in C^2(\R^{d+1};\R^m)$, and $\lim_{|x|\to\infty}\tilde v(x)=0$, as well as $\lim_{|x|\to\infty}\nabla \tilde v(x)=0$. \qed
\end{proof}

To complete the proof of Theorem \ref{th1ground}, it remains to establish the variational characterization of the solution $u$. Let $\tilde v(x)=u(x)-e(x_{d+1})$, $[\tilde V(y)](x_{d+1})=\tilde v(x_1,\ldots,x_{d+1})$, and let $\tilde u$ be a weak solution to system \eqref{system} such that $\tilde u(x)=e(x_{d+1})+\tilde \xi(x)$, with  $\tilde\xi(x)=[\tilde\Xi(y)](x_{d+1})$,  $\tilde\Xi\in\mathcal D^{1,2}_{rad}(\R^d;\mathcal H)$, $\tilde\Xi \not\equiv 0$, and $\mathcal W(\tilde\Xi)\in L^1(\R^d)$. Then, we have in view of \eqref{hamground3}:
\begin{subequations}\label{comb}
\begin{equation}
\int_{\R^{d}}\|\nabla\tilde V(y)\|^2_{\mathcal H^d}dy=-\frac{2d}{d-2}\int_{\R^d}\mathcal W(\tilde V(y))dy,
\end{equation}
\begin{equation}
\int_{\R^{d}}\|\nabla\tilde \Xi(y)\|^2_{\mathcal H^d}dy=-\frac{2d}{d-2}\int_{\R^d}\mathcal W(\tilde \Xi(y))dy.
\end{equation}
\end{subequations}
Moreover, we know that $\tilde V(y)=V(\mu^{-\frac{1}{2}}y)$, and the minimizer $V$ satisfies
\begin{align}\label{minpb1bis}
T=\int_{\R^{d}}\|\nabla V\|^2_{\mathcal H^d}=\min\left\lbrace \int_{\R^d}\dfrac{1}{2}\|\nabla U\|_{\mathcal H^d}^2 \, :U\in\mathcal A \right\rbrace, \text{ and } \int_{\R^d}\mathcal W( V)=-1.
\end{align}
Consequently, one can check that $$\frac{1}{2}\int_{\R^d}\|\nabla \tilde V\|^2_{\mathcal H^d}= T \Big(-\int_{\R^d}\mathcal W(\tilde V)\Big)^{\frac{d-2}{d}}.$$
On the other hand, we also have (cf. \eqref{eqsca}):
$$\frac{1}{2}\int_{\R^d}\|\nabla \tilde \Xi\|_{\mathcal H^d}^2\geq T \Big(-\int_{\R^d}\mathcal W(\tilde\Xi)\Big)^{\frac{d-2}{d}}.$$
Combining \eqref{comb} with these relations, we deduce that $-\int_{\R^d}\mathcal W(\tilde V)\leq -\int_{\R^d}\mathcal W(\tilde\Xi)$, which implies in view of \eqref{hamground4} that 
$\mathcal E(u)\leq\mathcal E( \tilde u)$. \qed
\end{proof}

Now, we provide examples of potentials satisfying the assumptions of Theorem \ref{th1ground}.
\begin{proposition}\label{propo1}
Let $\widetilde W\in C^2(\R^m;\R)$ be a nonnegative potential vanishing only at the points $a^-$ and $a^+$, and satisfying \eqref{w2}, as well as  $\liminf_{|u|\to\infty}\widetilde W(u)>0$. Let also $e\in C^2(\R;\R^m)$ be a minimizing heteroclinic orbit\footnote{We refer for instance to \cite{antonop} for the existence of such an orbit.} connecting $a^-$ to $a^+$:
\begin{subequations}
\begin{equation}
 e''(t)=\nabla \widetilde W(e(t)), \ \lim_{t\to\pm\infty}e(t)=a^\pm,
 \end{equation}
 \begin{equation}\label{min22}
\tilde E_1(e)=  \int_\R\Big[\frac{1}{2}| e'|^2+\widetilde W(e)\Big] \leq \tilde E_1(e+h)=\int_\R \Big[\frac{1}{2}| e'+h'|^2+\widetilde W(e+h)\Big], \forall h \in H^1(\R;\R^m).
  \end{equation}
  \end{subequations}
Then, there exists an open neighbourhood $\mathcal O$ of the compact set $K=e(\R)\cup\{a^\pm\}$, and a nonnegative potential $ W\in C^2(\R^m;\R)$ coinciding with $\widetilde W$ in $\mathcal O$, and satisfying assumptions \eqref{w13} and \eqref{ass2}.
\end{proposition}
\begin{proof}
Let $\gamma_0\in C^2([-1,1];\R^m)$ be a simple regular curve such that $\gamma_0(\pm1)=a^\pm$, and $\gamma_0((-1,1))\cap K=\emptyset$. Given $\epsilon>0$ small enough, let $\mathcal O$ be an open and bounded neighbourhood of $K$ such that $\mathcal O \cap \gamma_0([-1+\epsilon,1-\epsilon])=\emptyset$. We construct a nonnegative potential $ W$ in such a way that
\begin{itemize}
\item $W(u)=\widetilde W(u)$, for $u\in \mathcal O$,
\item  $0<W(u)\leq 2\sup_{ \gamma_0([-1,-1+\epsilon])\cup \gamma_0([1-\epsilon,1])}\widetilde W$, for $u\in \gamma_0((-1,1))$,
\item $W\in C^2(\R^m;\R)$, and the second derivatives of $W$ are bounded in $\R^m$.
\end{itemize}
By taking $\epsilon$ small enough we can ensure that $\int_{\gamma_0} \sqrt{2W}< \tilde E_1(e)=E_1(e)$.

On the other hand, we know by \cite[Proposition 2.1]{book}, that there is an equipartition parametrization of the curve $\gamma_0(s)$, $s\in [-1,1]$. That is, there is an orientation preserving diffeomorphism $\phi:\R \to(-1,1)$, such that the map $\gamma(t)=\gamma_0(\phi(t))$, satisfies $\frac{1}{2}|
\gamma'(t)|^2=W(\gamma(t))$, $\forall t\in\R$. Consequently, it follows that
$$E_1(
\gamma)= \int_\R\Big[\frac{1}{2}| \gamma'|^2+W(\gamma)\Big]=\int_\R|\gamma'|^2=\int_\R\sqrt{2W(\gamma)}|\gamma'|=\int_{\gamma_0} \sqrt{2W}<E_1(e),$$
and $|\gamma-a^-|^2$ (respectively $|\gamma-a^+|^2$) is integrable in a neighbourhood of $-\infty$ (respectively $+\infty$), since the zeros $a^\pm$ of $W$ are nondegenerate, and $W(\gamma)\in L^1(\R)$. This proves that $\gamma-e\in H^1(\R;\R^m)$, and condition  \eqref{negw} is fulfiled.

To complete the proof, it remains to check that \eqref{nondegw} also holds. Let $\beta>0$ be such that $\|h\|_{H^1(\R;\R^m)}^2\leq \beta\Rightarrow e(t)+h(t)\in\mathcal O$, $\forall t\in\R$. Then, when $\|h\|_{L^2(\R;\R^m)}^2\leq \frac{\beta}{1+2L}$, and $\|h'\|_{L^2(\R;\R^m)}^2\leq \frac{2L\beta}{1+2L}$, it is clear in view of \eqref{min22} that $\mathcal W(h)=E_1(e+h)-E_1(e)\geq 0$. Otherwise, when $\|h\|_{L^2(\R;\R^m)}^2\leq \frac{\beta}{1+2L}$, and $\|h'\|_{L^2(\R;\R^m)}^2\geq \frac{2L\beta}{1+2L}$, we have again $\mathcal W(h)\geq 0$  in view of \eqref{alter} and \eqref{abc}. This proves that
$\mathcal W(h)\geq 0$, provided that  $\|h\|_{L^2(\R;\R^m)}^2\leq \frac{\beta}{1+2L}$. \qed

\end{proof}

The following version of Theorem \ref{th1ground} applies when instead of \eqref{w3b}, we assume that
\begin{equation}\label{w3bb}
\exists R_0>0 \text{ such that } \nabla W(u)\cdot u>0, \ \text{ provided that } |u|\geq R_0.
\end{equation}
\begin{corollary}
Let $d\geq 3$. Under assumptions \eqref{w1}, \eqref{w2}, \eqref{ode0}, \eqref{ass2} and \eqref{w3bb}, there exists a solution $u\in C^2(\R^d\times\R;\R^m)$ to system \eqref{system} satisfying  properties (i) and (ii) of Theorem \ref{th1ground}. Property (iii) of Theorem \ref{th1ground} also holds if we assume in addition that the solution $\tilde u$ is bounded.
\end{corollary}
\begin{proof}
The maximum principle (cf. for instance \cite[Theorem 3.2]{ps}) implies that every bounded solution of \eqref{system} takes its values in the ball $B(0,R_0)$ of $\R^m$. This is  true in particular for the trivial solutions $a^\pm$, for the heteroclinic orbit $e$, and for the solution $\tilde u$ considered in Theorem \ref{th1ground} (iii). Moreover, the map $h_0$ in \eqref{negw} can be taken so that $(e+h_0)(\R)\subset B(0,R_0)$, since the projection onto the closed ball $\overline{B(0,R_0-\epsilon)}$ (with $\epsilon>0$ small enough), reduces the energy $E_1$.

Next, we modify the potential $W$ outside the ball $B(0,R_0)$ in such a way that both \eqref{w3b} and \eqref{w3bb} hold. In view of Theorem \ref{th1ground}, we obtain a bounded solution $u$ corresponding to the modified potential. However, it follows again from the maximum principle that $|u(x)|< R_0$, $\forall x \in \R^{d+1}$. Thus, $u$ is also a solution of \eqref{system} for the initial potential $W$. \qed
\end{proof}

 \subsection{Ground state for $d=2$}
The case $d=2$ is particular, since in view of assumption (\textbf{H2}) of Theorem \ref{th2}, the heteroclinic orbit $e$ must be a nondegenerate local minimum of $\mathcal W$. In this context we need to strengthen assumptions \eqref{w13} and \eqref{ass2}. When $d=2$, we assume that
\begin{subequations}\label{w13bis}
\begin{equation}\label{w1bis}
\text{$W\in C^{2}(\R^m; \R)$ is nonnegative, and has (at least) two zeros $a^-$ and $a^+$,}
\end{equation}
\begin{equation}\label{w1ter}
\text{$W$ is invariant by the reflection $\sigma$ which exchanges $a^\pm$,}
\end{equation}
\begin{equation}\label{w2bis}
\text{the Hessian matrices $D^2 W(a^\pm)$ are positive definite,}
\end{equation}
\begin{equation}\label{w3bbis}
\text{the map $u\mapsto \nabla W(u)$ is globally Lipschitz in $\R^m$.}
\end{equation}
Moreover, we assume the existence of a symmetric heteroclinic orbit $e\in C^2(\R;\R^m)$ connecting $a^-$ to $a^+$:
\begin{equation}\label{ode0bis}
e''(t)=\nabla W(e(t)), \ \lim_{t\to\pm\infty}e(t)=a^\pm, \, e(-t)=\sigma e(t), \, \forall t\in\R.
\end{equation}
\end{subequations}
Without loss of generality we assume that $e(0)=0$. We shall work in the Hilbert space
$$\mathcal H_\sigma=\{ h\in L^2(\R;\R^m): \, h(-t)=\sigma h(t) \text{ for a.e. } t\in\R\}\subset L^2(\R;\R^m),$$
of symmetric maps. The symmetry of $h\in\mathcal H_\sigma$ removes the degeneracy due to the translation
invariance of the functional $\mathcal W$.
Finally, we assume that
\begin{subequations}\label{ass2bis}
\begin{equation}\label{nondegwbis}
\text{$\mathcal W(h)\geq \alpha \|h\|^2_{\mathcal H_\sigma}$ holds for $h\in \mathcal H_\sigma$ such that $\|h\|_{\mathcal H_\sigma}\leq \delta$ (with $\alpha,\delta>0$ small enough),}
\end{equation}
\begin{equation}\label{negwbis}
\text{and $\mathcal W(h_0)<0$, for some $h_0\in H^1(\R;\R^m)\cap \mathcal H_\sigma$. }
\end{equation}
\end{subequations}

\begin{remark}\label{Rnondeg}
Assumption \eqref{nondegwbis} holds provided that the heteroclinic orbit $e$ is nondegenerate in the sense that $0$ is a simple eigenvalue of the linearized operator $T:W^{2,2}(\R;\R^m)\rightarrow L^2(\R;\R^m)$, $T\varphi=-\varphi^{\prime\prime}+D^2W(e)\varphi$ (cf. \cite{scha}, and in particular Lemma 4.5 therein).  Indeed, the nondegeneracy of $e$ implies the existence of constants $\alpha,\beta>0$ such that
\begin{equation}\label{min22bis}
h\in \mathcal H_{\sigma} \cap H^1(\R;\R^m), \  \|h\|^2_{H^1(\R;\R^m)}\leq\beta \Rightarrow \mathcal W(h)\geq \alpha  \|h\|_{H^1(\R;\R^m)}^2.
\end{equation}
Then, when $\|h\|_{L^2(\R;\R^m)}^2\leq \frac{\beta}{2L+2\alpha+1}$, and $\|h'\|_{L^2(\R;\R^m)}^2\leq \frac{(2L+2\alpha)\beta}{2L+2\alpha+1}$, it is clear in view of \eqref{min22bis} that $\mathcal W(h)\geq  \alpha  \|h\|_{L^2(\R;\R^m)}^2$. Otherwise, when $\|h\|_{L^2(\R;\R^m)}^2\leq \frac{\beta}{2L+2\alpha+1}$, and $\|h'\|_{L^2(\R;\R^m)}^2\geq \frac{(2L+2\alpha)\beta}{2L+2\alpha+1}$, we have again $\mathcal W(h)\geq  \alpha  \|h\|_{L^2(\R;\R^m)}^2$  in view of \eqref{alter} and \eqref{abc}. This proves that \eqref{nondegwbis} holds with $\delta=\big(\frac{\beta}{2L+2\alpha+1}\big)^{\frac{1}{2}}$. Finally, we refer to \cite{jen} and \cite[Theorem 4.3 and Remark 4.4]{scha} for examples of potentials having such nondegenerate orbits.
\end{remark}

In Proposition \ref{propo1bis}, we provide examples of potentials $W$ satisfying assumptions \eqref{w13bis} and \eqref{ass2bis}.
When $d=2$, and \eqref{w13bis} as well as \eqref{ass2bis} hold, it is clear that the potential $\mathcal W:\mathcal H_\sigma\to (-\infty,+\infty]$ satisfies the assumptions (\textbf{H1})-(\textbf{H4}) of Theorem \ref{th2} applied to the Hilbert space $\mathcal H_\sigma$. Therefore, we deduce the existence of a radial map $V\in\mathcal A$ solving the minimization problem \eqref{minpb2}, in the class
\begin{equation}\label{class2bis}
\mathcal A=\left\lbrace U\in H^{1}_{rad}(\R^2;\mathcal H_\sigma) :   U\neq 0, \ \mathcal W(U)\in L^1(\R^2), \ \ \int_{\R^2}\mathcal W(U)\leq 0 \right\rbrace.
\end{equation}
Moreover, we have $\int_{\R^2}\mathcal W(V)=0$, and $V\not\equiv 0$.  From the minimizer $V$, we recover a solution $u$ to system \eqref{system} satisfying the properties of Theorem \ref{th2ground} below. In what follows, we write $x=(x_1,x_2, x_3)=(y,x_3)\in\R^3$, with $y=(x_1,x_2)\in\R^2$.

\begin{theorem}\label{th2ground}
Under assumptions \eqref{w13bis} and \eqref{ass2bis}, there exists a solution $u\in C^2(\R^3;\R^m)$ to system \eqref{system} such that
\begin{itemize}
\item[(i)]  $u(y,x_3)=u(z,x_3)$, whenever $|y|=|z|$ (with $,y,z\in \R^2$), and  $u(y,-x_3)=\sigma u(y,x_3)$, $\forall (y,x_3)\in\R^2\times\R$.
\item[(ii)] $\lim_{|x|\to\infty}(u(x)-e(x_3))=0$ (with $x=(x_1,x_2, x_3)$), and $u(x)-e(x_3)\not\equiv 0$. In particular $\lim_{y\in\R^2, |y|\to\infty}u(y,x_3)=e(x_3)$.
\item[(iii)] Let $\tilde u$ be a weak solution to system \eqref{system} such that $\tilde u(x)=e(x_{3})+\tilde \xi(x)$, with  $\tilde\xi(x)=[\tilde\Xi(y)](x_{3})$, $\tilde\Xi\in H^{1,2}_{rad}(\R^2;\mathcal H_\sigma)$, $\tilde\Xi \not\equiv 0$, and $\mathcal W(\tilde\Xi)\in L^1(\R^2)$. Then, we have $\mathcal E(u)\leq\mathcal E( \tilde u)$, where $\mathcal E$ is defined in \eqref{renormene}.
\end{itemize}
\end{theorem}
\begin{proof}

Let $V\in\mathcal A$, be the minimizer provided by Theorem \ref{th2}. We proceed as in the proof of Theorem \ref{th1ground}, to recover the solution $u$  from the minimizer $V$.

Let $H^1_{equ}(\R^3;\R^m)$ be the subspace of equivariant maps $w \in  H^1(\R^3;\R^m)$ which are
\begin{itemize}
\item radial with respect to $y$, i.e. $ w( gy,x_3)=w(y,x_3)$, $\forall g \in O(\R^2)$, and for a.e. $(y,x_3)\in \R^3$,
\item symmetric with respect to $x_3$, i.e. $w(y,-x_3)=\sigma w(y,x_3)$, for a.e. $(y,x_3)\in \R^3$.
\end{itemize}

By adjusting the proof of Lemma \ref{recover}, we first establish the following result.
\begin{lemma}\label{recoverbis}
 Let $v(x):=[V(y)](x_3)$. Then, we have $v\in H^1_{equ}(\R^3;\R^m)$. On the other hand, let $\phi\in  H^1_{equ}(\R^3;\R^m)$, and let $[\Phi(y)](x_3):=\phi(x)$. Then, for every $\lambda\in\R$, $V+\lambda\Phi$ belongs to $H^1_{rad}(\R^2; \mathcal H_\sigma)$, and $\mathcal W(V+\lambda\Phi)\in L^1(\R^2)$.
\end{lemma}

In view of Lemma \ref{recoverbis}, $v$ is a minimizer of the functional $$S(w)=\int_{\R^3}|\nabla_y w(x)|^2 dx,$$ in the class $$\mathcal C=\{w=v+\phi: \phi \in  H^1_{equ}(\R^3;\R^m): w\not\equiv 0, \, F(w)\leq 0\},$$
where $$F(w)=\int_{\R^3} \Big(\frac{1}{2}|w_{x_3}(x)|^2+f(x_3,w(x))\Big)dx=\int_{\R^2} \mathcal W(V+\Phi)dy.$$
Next, we define the continuous linear functionals on $H^1_{equ}(\R^3;\R^m)$:
$$G_1(\phi)=\int_{\R^3}\big(v_{x_3}(x)\cdot\phi_{x_3}(x)+\big[\nabla W(e(x_3)+v(x))-\nabla W(e(x_3))\big]\cdot\phi(x)\big) dx,$$
and $$G_2(\phi)=\int_{\R^3}(v_{x_1}\cdot \phi_{x_1}+v_{x_2}\cdot \phi_{x_2} )dx,$$
where $\phi\in H^1_{equ}(\R^3;\R^m)$. One can see that
$$\frac{d}{d\lambda}F(v+\lambda\phi)\Big|_{\lambda=0}=G_1(\phi).$$

At this stage, we prove the analog of Lemma \ref{recover2}:
\begin{lemma}\label{recover2bis}(Lagrange multiplier)
\begin{itemize}\item[(i)] We have $G_2\not\equiv 0$.
\item[(ii)] There exists a constant $\mu>0$ such  that $\mu G_2+G_1=0$.
\end{itemize}
\end{lemma}
\begin{proof}
(i) Assuming by contradiction that $G_1\equiv 0$, then for a.e. $y\in\R^2$, the map $V(y):x_3\mapsto v(y,x_3)$ belongs to $H^1(\R;\R^m)$, and the map $\R\ni x_3\mapsto e(x_3)+v(y,x_3)$ is a heteroclinic orbit connecting $a^-$ to $a^+$.
In addition, since every heteroclinic orbit converges at $\pm\infty$ exponentially fast  to $a^\pm$, it follows that $V(y)\in H^2(\R;\R^m)$.
Consequently, writing by abuse of notation $V(r)$ instead of $V(y)$, where $r = |y|$, $y\in\R^2$, there exists a negligible set $N\subset(0,\infty)$, such that  for $r\in (0,\infty)\setminus N$:
\begin{itemize}
\item $V(r)\in H^2(\R;\R^m)$, and $\R\ni x_3\mapsto e(x_3)+[V(r)](x_3)$ is a heteroclinic orbit connecting $a^-$ to $a^+$: $e''(x_3)+\frac{d^2 [V(r)]}{d x_3^2}(x_3)=\nabla W(e(x_3)+[V(r)](x_3))$, $\forall x_3\in\R$.
\item $\lim_{s\to r}\frac{V(r+s)-V(r)}{s-r}=V_r(r)$ holds in $L^2(\R;\R^m)$, where $V_r$ is the weak derivative of the map $r\mapsto V(r)$.
\end{itemize}
Next, we introduce the  function $(0,\infty)\ni r \mapsto \zeta(r)=\int_\R W\big(e(x_3)+[V(r)](x_3)\big)d x_3$. By the mean value theorem we have
\begin{equation}\label{proceed}
\zeta(s)-\zeta(r)=\int_\R p_{r,s}(x_3)\cdot\big([V(s)](x_3)-[V(r)](x_3)\big)dx_ 3,
\end{equation}
with $p_{r,s}(x_3)=\nabla W \big(e(x_3)+[V(r)](x_3)+c_{r,s}(x_3) ([V(s)](x_3)-[V(r)](x_3) ) \big)$, and $0\leq c_{r,s}(x_3)\leq 1$. Moreover, one can see using \eqref{w3bbis} that the functions $p_{r,s}$ are uniformly bounded in $L^2(\R;\R^m)$, provided that $r,s\in [a,b]$, with $0<a<b<\infty$.
This implies that the function $\zeta$ is absolutely continuous on $[a,b]$. We also compute for $r\in (0,\infty)\setminus N$:
\begin{align}\label{comb1}
\lim_{s\to r}\frac{\zeta(s)-\zeta(r)}{s-r}&=\lim_{s\to r}\int_\R p_{r,s}(x_3)\cdot\frac{[V(s)](x_3)-[V(r)](x_3)}{s-r}dx_ 3 \nonumber \\
&=\int_\R \nabla W \big(e(x_3)+[V(r)](x_3)\big)\cdot [V_r(r)](x_3)dx_ 3,
\end{align}
since in view of \eqref{w3bbis}, $p_{r,s}\to \nabla W (e+[V(r)])$ in $L^2(\R;\R^m)$, as $s\to r$.

On the other hand, for every $r \in(0,\infty)\setminus N$, the heteroclinic orbit $\R\ni x_3\mapsto e(x_3)+[V(r)](x_3)$ satisfies the equipartition relation:
$$W\big(e(x_3)+[V(r)](x_3)\big)=\frac{1}{2}\Big|e'(x_3)+\frac{d [V(r)]}{dx_3}(x_3)\Big|^2, \, \forall x_3\in\R,$$
thus we have for $r,s \in(0,\infty)\setminus N$:
\begin{align*}
\zeta(s)-\zeta(r)&=\frac{1}{2}\int_\R \Big[\Big| e'+\frac{ d[V(s)]}{dx_3}\Big|^2-\Big|e'+ \frac{ d[V(r)]}{dx_3}\Big|^2\Big]dx_3
\\
&=\frac{1}{2}\int_\R \Big[\frac{ d([V(s)]-[V(r)])}{dx_3}\Big]\cdot \Big[2 e'+ \frac{ d([V(s)]+[V(r)])}{dx_3}\Big]dx_3\\
&=-\frac{1}{2}\int_\R [V(s)-V(r)]\cdot [\nabla W(e+V(s))+\nabla W(e+V(r))]dx_3,
\end{align*}
and
\begin{equation}\label{comb2}
\lim_{(0,\infty)\setminus N\ni s\to r}\frac{\zeta(s)-\zeta(r)}{s-r}=-\int_\R \nabla W \big(e(x_3)+[V(r)](x_3)\big)\cdot [V_r(r)](x_3)dx_ 3,
\end{equation}
since using \eqref{w3bbis}, one can see that $ \nabla W(e+V(s))\to\nabla W(e+V(r))$ in $L^2(\R;\R^m)$, as $s\to r$.
In view of \eqref{comb1} and \eqref{comb2}, we deduce that $\lim_{s\to r}\frac{\zeta(s)-\zeta(r)}{s-r}=0$, for $r \in(0,\infty)\setminus N$, and consequently the function $\zeta$ is constant. Moreover, proceeding as in \eqref{proceed}, and writing
\begin{equation}\label{proceedbis}
\zeta(r)-\int_\R W(e(x_3))dx_3=\int_\R p_{r}(x_3)\cdot[V(r)](x_3)dx_ 3,
\end{equation}
with $p_{r}(x_3)=\nabla W \big(e(x_3)+c_{r}(x_3)[V(r)](x_3)  \big)$, and $0\leq c_{r}(x_3)\leq 1$, one can see as previously
that $\lim_{r\to\infty}\zeta(r)=\int_\R W(e)$. Thus, $\zeta(r)\equiv \int_\R W(e)$ . In view of the equipartition relation, this implies that $\mathcal W(V(r))=0$, $\forall r \in (0,\infty)\setminus N$. Therefore, using \eqref{nondegwbis}, we conclude  that $V\equiv 0$, which is a contradiction.

(ii) It is clear that $G_2\not\equiv 0$, since $G_2(v)=\int_{\R^2}\|\nabla V(y)\|^2_{\mathcal H_\sigma^2}dy >0$. Moreover, if $G_1(\phi)<0$, we have $v+\lambda \phi\in \mathcal C$, for $0<\lambda\ll 1$, and thus $G_2(\phi)\geq 0$. Consequently, there exists a constant $\mu>0$ such  that $\mu G_2+G_1=0$. \qed

\end{proof}
To complete the proof of Theorem \ref{th2ground}, we proceed exactly as in the higher dimensional case. One can see that the map $v$ solves \eqref{prel}, while the map $u(x)=\tilde v(x)+e(x_{3})$, with $\tilde v(x)=v(\mu^{-\frac{1}{2}}y,x_{3})$, is a solution to system \eqref{system}. On the other hand, the asymptotic properties of $u$ are established as in Lemma \ref{AS}. Finally,  let $\tilde u$ be a weak solution to system \eqref{system} such that $\tilde u(x)=e(x_{3})+\tilde \xi(x)$, with  $\tilde\xi(x)=[\tilde\Xi(y)](x_{3})$,  $\tilde\Xi\in H^{1,2}_{rad}(\R^2;\mathcal H_\sigma)$, $\tilde\Xi \not\equiv 0$, and $\mathcal W(\tilde\Xi)\in L^1(\R^2)$. Then, setting $[\tilde V(y)](x_{3})=\tilde v(x_1,x_2,x_3)$, we have
\begin{align*}
\mathcal E(u)=\frac{1}{2}\int_{\R^{2}}\|\nabla\tilde V\|^2_{\mathcal H^2}=\frac{1}{2}\int_{\R^{2}}\|\nabla V\|^2_{\mathcal H^2}\leq \frac{1}{2}\int_{\R^{2}}\|\nabla\tilde \Xi\|^2_{\mathcal H^2}=\mathcal E(\tilde u),
\end{align*}
since $\tilde V(y)=V(\mu^{-\frac{1}{2}}y)$, and $\int_{\R^2}\mathcal W(\tilde\Xi)=\int_{\R^2}\mathcal W(\tilde V)=0$, in view of \eqref{hamground3}. \qed

\end{proof}

Now, we provide examples of potentials satisfying the assumptions of Theorem \ref{th2ground}.
\begin{proposition}\label{propo1bis}
Let $\widetilde W\in C^2(\R^m;\R)$ be a nonnegative potential vanishing only at the points $a^-$ and $a^+$, and satisfying \eqref{w1ter}, \eqref{w2bis} as well as $\liminf_{|u|\to\infty}\widetilde W(u)>0$. Let also $e\in C^2(\R;\R^m)$ be a nondegenerate symmetric minimizing heteroclinic orbit\footnote{We refer to Remark \ref{Rnondeg} and to \cite{antonop} for the existence of such an orbit.} connecting $a^-$ to $a^+$:
\begin{subequations}
\begin{equation}
 e''(t)=\nabla \widetilde W(e(t)), \ \lim_{t\to\pm\infty}e(t)=a^\pm,  \, e(-t)=\sigma e(t), \, \forall t\in\R,
 \end{equation}
 \begin{equation}\label{min22ter}
  \int_\R\Big[\frac{1}{2}| e'|^2+\widetilde W(e)\Big] \leq\int_\R \Big[\frac{1}{2}| e'+h'|^2+\widetilde W(e+h)\Big], \forall h \in H^1(\R;\R^m),
  \end{equation}

\begin{equation}
\text{$\dim\ker T=1$, where $T:W^{2,2}(\R;\R^m)\rightarrow L^2(\R;\R^m)$, $T\varphi=-\varphi^{\prime\prime}+D^2\widetilde W(e)\varphi$.}
 \end{equation}
  \end{subequations}
Then, there exists an open neighbourhood $\mathcal O$ of the compact set $K=e(\R)\cup\{a^\pm\}$, and a nonnegative potential $ W\in C^2(\R^m;\R)$ coinciding with $\widetilde W$ in $\mathcal O$, and satisfying assumptions \eqref{w13bis} and \eqref{ass2bis}.
\end{proposition}
\begin{proof}
We consider a simple regular symmetric curve $\gamma_0\in C^2([-1,1];\R^m)$ such that $\gamma_0(\pm1)=a^\pm$, and $\gamma_0((-1,1))\cap K=\emptyset$. Then, we proceed as in the proof of Proposition \ref{propo1}. \qed
\end{proof}

As in the higher dimensional case, we can show that the following version of Theorem \ref{th2ground} applies when instead of \eqref{w3bbis}, we assume that \eqref{w3bb} holds.

\begin{corollary}
Under assumptions \eqref{w1bis}, \eqref{w1ter}, \eqref{w2bis}, \eqref{ode0bis}, \eqref{ass2bis} and \eqref{w3bb}, there exists a solution $u\in C^2(\R^3;\R^m)$ to system \eqref{system} satisfying properties (i) and (ii) of Theorem \ref{th2ground}. Property (iii) of Theorem \ref{th2ground} also holds if we assume in addition that the solution $\tilde u$ is bounded.
\end{corollary}

\appendix
\section{Hamiltonian and Pohozaev identities}
Hamiltonian identities have originally been derived in \cite{gui} for solutions of system \eqref{system0} satisfying some asymptotic conditions. Here, we establish a version of these identities for the solutions of section \ref{sec:application}. They are used to prove the ground state property (iii) of Theorems \ref{th1ground} and \ref{th2ground}.

\begin{lemma}\label{hamid} Let $d\geq 2$, and set $x=(x_1,\ldots,x_{d+1})=(y,x_{d+1})$, with $y=(x_1,\ldots,x_d)$. Under the assumptions of Theorem \ref{th1ground} when $d\geq 3$ (respectively Theorem \ref{th2ground} when $d=2$), let $\tilde u$ be a weak solution to system \eqref{system} such that $\tilde u(x)=e(x_{d+1})+\tilde \xi(x)$, with
\begin{equation*}
  \begin{cases}
    \tilde\Xi\in\mathcal D^{1,2}_{rad}(\R^d;\mathcal H), \ \tilde\Xi \not\equiv 0, \text{ and } \mathcal W(\tilde\Xi)\in L^1(\R^d), &\text{ when } d\geq 3, \\
    \tilde\Xi\in H^{1,2}_{rad}(\R^2;\mathcal H_\sigma), \ \tilde\Xi \not\equiv 0, \text{ and } \mathcal W(\tilde\Xi)\in L^1(\R^2), &\text{ when } d=2.
  \end{cases}
\end{equation*}
Then, for every $i=1,\ldots, d$, there exists a negligible set $N_i\subset\R$ such that
\begin{align}\label{hamground}
\int_{y_i=\lambda}\Big(\frac{1}{2}\|\tilde \Xi_{y_i}(y)\|^2_{\mathcal H}-\frac{1}{2}\sum_{j\neq i}\|\tilde \Xi_{y_j}(y)\|^2_{\mathcal H}-\mathcal W(\tilde\Xi(y))\Big)dy_1\ldots dy_{i-1}dy_{i+1}\ldots dy_d=0
\end{align}
holds for $\lambda\in\R\setminus N_i$. Consequently, we also have
\begin{align}\label{hamground2}
\int_{\R^d}\Big(\frac{1}{2}\|\tilde \Xi_{y_i}(y)\|^2_{\mathcal H}-\frac{1}{2}\sum_{j\neq i}\|\tilde \Xi_{y_j}(y)\|^2_{\mathcal H}-\mathcal W(\tilde\Xi(y))\Big)dy=0, \ \forall i=1,\ldots d,
\end{align}
as well as
\begin{align}\label{hamground3}
\int_{\R^{d}}\Big(\big(1-\frac{d}{2}\big)\|\nabla\tilde \Xi(y)\|^2_{\mathcal H^d}-d\mathcal W(\tilde\Xi(y))\Big)dy=0.
\end{align}
Moreover,
\begin{align}\label{hamground4}
\mathcal E(\tilde u)=\frac{1}{d}\int_{\R^{d}}\|\nabla\tilde \Xi(y)\|^2_{\mathcal H^d}, \text{ and if  $d\geq3$, }\mathcal E(\tilde u)=-\frac{2}{d-2}\int_{\R^d}\mathcal W(\tilde\Xi(y))dy,
\end{align}
where $\mathcal E$ is defined in \eqref{renormene}.
\end{lemma}
\begin{proof}
By reproducing the arguments of Lemmas \ref{recover} and \ref{AS} when $d\geq 3$ (respectively Lemmas \ref{recoverbis} and \ref{AS} when $d=2$), we first notice that $\tilde u$ is a classical solution satisfying $\lim_{|x|\to\infty}\tilde \xi(x)=0$, and $\lim_{|x|\to\infty}\nabla \tilde \xi(x)=0$. In what follows, we still denote by $e$, the map $(x_1,\ldots, x_{d+1})\mapsto e(x_{d+1})$. Next, we consider the stress-energy tensor (cf. \cite[Chapter 3]{book}) associated to solutions of system \eqref{system}:
\begin{equation}\label{tensor}
  T_{i, j}=\tilde u_{x_i}\cdot \tilde u_{x_j}-\delta_{ij}\Big(\frac{1}{2}|\nabla \tilde u|^2+W(\tilde u)\Big), \ 1\leq i,j\leq d+1,
\end{equation}
and the divergence free vector fields $A_i(x)=(T_{i,1},\ldots, T_{i,d+1})\in\R^{d+1}$, ($i=1,\ldots,d$).
Given $i\in\{1,\ldots,d\}$, there exits a negligible set $N_i\subset\R$ such that the functions $\|\nabla \tilde \Xi(y)\|^2_{\mathcal H^d}$ and $\mathcal W(\tilde \Xi(y))$ are integrable on the hyperplanes $y_i=\lambda$, for every $\lambda\in\R\setminus N_i$. By applying the divergence theorem to the vector field $A_i$ in the box
$$B=[-R,R]^{i-1}\times [\lambda^-\lambda^+]\times[-R,R]^{d-i}\times [-L,L],$$ with $\lambda^\pm\in\R\setminus N_i$, we obtain the equation $$\sum_{j=1}^{d+1} I_j^+-\sum
_{j=1}^{d+1} I_j^-=0,$$ where
\begin{equation*}
\begin{cases}
\partial B=\cup_{j=1}^{d+1} B_{j}^\pm, \\
B_i^\pm=\{x\in \R^{d+1}: x_i=\lambda^\pm,  | x_{d+1}|\leq L, |x_k|\leq R, \forall k \notin \{i,d+1\}\},\\
B_{d+1}^\pm=\{ x\in \R^{d+1}: x_i\in [\lambda^-,\lambda^+], x_{d+1}=\pm L, |x_k|\leq R, \forall k \notin \{i,d+1\}\},\\
\text{while for } j\notin\{i,d+1\},\\
B_j^\pm=\{ x\in \R^{d+1}: x_j =\pm R,\, x_i\in [\lambda^-,\lambda^+], |x_{d+1}|\leq L, |x_k|\leq R, \forall k \notin \{i,j,d+1\}\},
\end{cases}
\end{equation*}
and
\begin{equation*}
I_{i}^\pm=\int_{B_i^\pm}  \big(\frac{1}{2}|\tilde\xi_{x_i}|^2-\big(\sum_{\{1,\ldots,d+1\}\ni l\neq i}\frac{1}{2}|\tilde\xi_{x_l}|^2\big)- \frac{1}{2}|e_{x_{d+1}}|^2-e_{x_{d+1}}\cdot \tilde\xi_{x_{d+1}}-W(e+\tilde\xi)\big),
\end{equation*}

\begin{equation*}
\forall j\notin\{i,d+1\}: \ I_{j}^\pm=\int_{B_j^\pm}\tilde\xi_{x_i}\cdot\tilde\xi_{x_j},
\end{equation*}

\begin{equation*}
I_{d+1}^\pm=\int_{B_{d+1}^\pm}\tilde\xi_{x_i}\cdot(e_{x_{d+1}}+\tilde\xi_{x_{d+1}}).
\end{equation*}
It is clear that as $L\to\infty$, and $R$ remains fixed, $\lim_{L\to\infty}I_{d+1}^\pm=0$. On the other hand, after an integration by parts, we obtain
\begin{multline*}
I_{i}^\pm=\int_{B_i^\pm} \big( \frac{1}{2}|\tilde\xi_{x_i}|^2- \big(\sum_{\{1,\ldots,d+1\}\ni l\neq i}\frac{1}{2}|\tilde\xi_{x_l}|^2\big)-f(x_{d+1},\tilde\xi) -|e_{x_{d+1}}|^2\big)\\
-\int_{x_i=\lambda^\pm,   x_{d+1}= L, |x_k|\leq R, \forall k \notin \{i,d+1\}}\tilde\xi\cdot e_{x_{d+1}}+\int_{x_i=\lambda^\pm,   x_{d+1}=- L, |x_k|\leq R, \forall k \notin \{i,d+1\}}\tilde\xi\cdot e_{x_{d+1}},
\end{multline*}
and letting $L\to\infty$, one can see that
\begin{multline*}
\lim_{L\to\infty}(I_{i}^+-I_i^-)=\int_{y_i=\lambda^+,  |y_k|\leq R, \forall k \neq i} \big( \frac{1}{2}\|\tilde\Xi_{y_i}(y)\|_{\mathcal H}^2- \frac{1}{2}\sum_{\{1,\ldots,d\} \ni k\neq i}\|\tilde\Xi_{y_k}(y)\|_{\mathcal H}^2-\mathcal W(\tilde\Xi(y))\big) \\
-\int_{y_i=\lambda^-,  |y_k|\leq R, \forall k \neq i} \big( \frac{1}{2}\|\tilde\Xi_{y_i}(y)\|_{\mathcal H}^2- \frac{1}{2}\sum_{\{1,\ldots,d\} \ni k\neq i}\|\tilde\Xi_{y_k}(y)\|_{\mathcal H}^2-\mathcal W(\tilde\Xi(y))\big) .
\end{multline*}
while
\begin{equation*}
\forall l\in\{1,\ldots d\}\setminus \{i\}: \ \lim_{L\to\infty} I_{l}^\pm=\int_{y_l =\pm R,\, y_i\in [\lambda^-,\lambda^+], |y_k|\leq R, \forall k \notin \{i,l\}}\big(\tilde\Xi_{y_i}(y)\cdot\tilde\Xi_{y_l}(y)\big).
\end{equation*}
Finally, since $\|\nabla \tilde \Xi\|_{\mathcal H^d}^2\in L^1(\R^d)$, and $\mathcal W(\tilde \Xi )\in L^1(\R^d)$, we conclude letting $R\to\infty$ along a subsequence that
\begin{equation*}
\lim_{R\to\infty} \int_{y_l =\pm R,\, y_i\in [\lambda^-,\lambda^+], |y_k|\leq R, \forall k \notin \{i,l\}}\big(\tilde\Xi_{y_i}(y)\cdot\tilde\Xi_{y_l}(y)\big)=0,
\end{equation*}
while
\begin{multline*}
\lim_{R\to\infty}\int_{y_i=\lambda^\pm,  |y_k|\leq R, \forall k \neq i} \big( \frac{1}{2}\|\tilde\Xi_{y_i}(y)\|_{\mathcal H}^2- \frac{1}{2}\sum_{\{1,\ldots,d\}\ni k\neq i}\|\tilde\Xi_{y_k}(y)\|_{\mathcal H}^2-\mathcal W(\tilde\Xi(y))\big) \\
=\int_{y_i=\lambda^\pm} \big( \frac{1}{2}\|\tilde\Xi_{y_i}(y)\|_{\mathcal H}^2- \frac{1}{2}\sum_{\{1,\ldots,d\}\ni k\neq i}\|\tilde\Xi_{y_k}(y)\|_{\mathcal H}^2-\mathcal W(\tilde\Xi(y))\big) .
\end{multline*}
This proves that for every $\lambda^-,\lambda^+\in\R\setminus N_i$, we have
\begin{multline*}
\int_{y_i=\lambda^-}\Big(\frac{1}{2}\|\tilde \Xi_{y_i}(y)\|^2_{\mathcal H}-\frac{1}{2}\sum_{j\neq i}\|\tilde \Xi_{y_j}(y)\|^2_{\mathcal H}-\mathcal W(\tilde\Xi(y))\Big)\\
=\int_{y_i=\lambda^+}\Big(\frac{1}{2}\|\tilde \Xi_{y_i}(y)\|^2_{\mathcal H}-\frac{1}{2}\sum_{j\neq i}\|\tilde \Xi_{y_j}(y)\|^2_{\mathcal H}-\mathcal W(\tilde\Xi(y))\Big),
\end{multline*}
and using again the integrability of the functions $\|\nabla \tilde \Xi(y)\|^2_{\mathcal H^d}$ and $\mathcal W(\tilde \Xi(y))$, it follows that
\begin{multline}\label{hamin}
\forall \lambda\in\R\setminus N_i: \ \int_{y_i=\lambda}\Big(\frac{1}{2}\|\tilde \Xi_{y_i}(y)\|^2_{\mathcal H}-\frac{1}{2}\sum_{j\neq i}\|\tilde \Xi_{y_j}(y)\|^2_{\mathcal H}-\mathcal W(\tilde\Xi(y))\Big)=0.
\end{multline}
Next, by integrating \eqref{hamin} with respect to $y_i\in\R$, we derive \eqref{hamground2}. Finally, the sum of equations \eqref{hamground2} for $i=1,\ldots,d$ gives \eqref{hamground3}. \qed
\end{proof}

\section*{Acknowledgements}
P. S. would like to thank Jacopo Schino for introducing him to some of the techniques used in this paper. This research project is implemented within the framework of H.F.R.I call
``Basic research Financing (Horizontal support of all Sciences)'' under the National Recovery and Resilience Plan ``Greece 2.0'' funded by the European Union - NextGenerationEU (H.F.R.I. Project Number: 016097).


\bibliographystyle{elsarticle-num} 
\bibliography{yourbibfile} 

\end{document}